\documentclass[10pt, A4paper]{amsart}
\usepackage[a4paper, total={14.5cm, 24cm}]{geometry}

\usepackage{amssymb,bbm,amsmath}

\usepackage{graphicx}

\newcommand{\rotatesupseteq}{\rotatebox[origin=c]{90}{$\supseteq$}}

\usepackage{amsthm}
\usepackage{enumitem}
\usepackage{dsfont}
\usepackage{color}
\usepackage[unicode]{hyperref}
\usepackage[utf8]{inputenc}
\usepackage[T1]{fontenc}
\usepackage{geometry}
\usepackage{todonotes}
\usepackage{lmodern}
\usepackage{anyfontsize}
\usepackage{cleveref}
\usepackage{comment}
\usepackage{makecell}

\theoremstyle{plain}
\newtheorem{theorem}{Theorem}[section]
\newtheorem{proposition}[theorem]{Proposition}
\newtheorem{lemma}[theorem]{Lemma}
\newtheorem{corollary}[theorem]{Corollary}

\theoremstyle{definition}
\newtheorem{definition}[theorem]{Definition}
\newtheorem{remark}[theorem]{Remark}

\newtheorem{example}[theorem]{Example}

\newcommand{\R}{\mathbb{R}}
\newcommand{\N}{\mathbb{N}}

\newcommand{\cW}{\mathcal{W}}

\newcommand{\id}{{\rm id}}
\newcommand{\proj}{{\rm proj}}

\newcommand{\KRmonotone}{{triangular increasing}}
\newcommand{\KRregular}{{triangular regular}}

\newcommand{\KR}{\mathcal{KR}}
\newcommand{\AW}{\mathcal{AW}}
\newcommand{\W}{\mathcal{W}}

\newcommand{\cplbc}{\Pi_{\rm bc}}
\newcommand{\cpl}{\Pi}

\renewcommand{\P}{\mathcal P}

\newcommand{\kr}{\text{kr}}

\begin{document}

\title{The Knothe--Rosenblatt distance and its induced topology}

\author{M.\ Beiglböck, G.\ Pammer, and A.\ Posch}\thanks{Financial support through FWF-projects Y0782 and P35197 is gratefully acknowledged.}
\maketitle

\begin{abstract}
A basic and natural coupling between two probabilities on $\R^N$ is given by the Knothe--Rosenblatt coupling. It represents a  multiperiod extension of the quantile coupling and  is simple to calculate numerically.

We consider the  distance on $\mathcal P (\R^N)$ that is induced by considering the transport costs associated to the Knothe--Rosenblatt coupling. 
We show that this Knothe--Rosenblatt distance metrizes the \emph{adapted weak topology} which is a stochastic process version of the usual weak topology and plays an important role, e.g.\ concerning questions on stability of stochastic control and probabilistic operations. 
We also establish that the Knothe--Rosenblatt distance is a geodesic distance, give a Skorokhod representation theorem for the adapted weak topology, and provide multi-dimensional versions of our results.

\medskip

\noindent{\it keywords: Knothe--Rosenblatt coupling, triangular transformations, adapted weak topology, barycenter of probabilities}
\end{abstract} 
\section{Introduction}

We fix $N\in \N$ and write $\lambda$ for the Lebesgue measure on $(0,1)$. For $p\in [1,\infty)$ we equip  $L_p(\lambda^N)=L_p(\lambda^N; \R^N)$  with the usual norm and the set $\mathcal P _p(\R^N)$ of probabilities on $\R^N$ with finite $p$-th moments with $p$-Wasserstein distance inducing $p$-weak convergence. 
For $p=0$ we equip $L_0(\lambda^N)$ with  $\|f-g\|_0=\int |f-g|_0 \, d\lambda^N$ where $|\cdot|_0 := |\cdot|\wedge 1$ and consider usual weak convergence on the set  $\mathcal P_0(\R^N)$ of  probabilities on $\R^N$.

\subsection{Quantile processes}

The Knothe--Rosenblatt coupling is an $N$-step extension of the quantile coupling which we first recall:
For every probability $\mu $ on $\R$ there is an a.s.\ unique increasing \emph{quantile function} $Q=Q^\mu\colon(0,1)\to \R$ such that $Q^\mu_\#(\lambda)=\mu$. The \emph{quantile coupling} or  (\emph{co-monotone coupling}) of  probabilities $\mu, \nu$ is given by $(Q^\mu,Q^\nu)_\#(\lambda)$.
In multiple steps, we are interested in functions that are component wise increasing: 
\begin{definition} \label{def:quantile_process}
A map  $T=(T_k)_{k=1}^N\colon(0,1)^N\to \R^N$ is \emph{adapted} or \emph{triangular}  if each $T_k $ depends only on the first $k$ variables, i.e.\ $T_k(x_1,\ldots, x_N)=T_k(x_1,\ldots, x_k)$, (cf.\ \cite{BoKoMe05}).
A triangular $Q=(Q_k)_{k=1}^N: (0,1)^N\to \R^N$  is a \emph{quantile process} if for all $k\leq m \leq N, u_1, \ldots, u_m, u_k' \in (0,1) $ with $u_k \leq u'_k$ 
\begin{align}\label{CanonicalIncreasingI} 
 Q_k(u_1, \ldots,  u_k) \ \  & \leq \ \   Q_k(u_1, \ldots, u_k') \tag{\text{inc}} \\
 Q_k(u_1, \ldots, u_k) = Q_k(u_1, \ldots, u_k')  &    \Rightarrow    Q_m(u_1, \ldots, u_k, \ldots, u_m )= Q_m(u_1, \ldots, u_k', \ldots, u_m). \tag{\text{con}}  \label{CanonicalIncreasingII}
\end{align}

\end{definition}
\begin{lemma}\label{lem:Qmu}
    For $\mu\in \P_0(\R^N)$ there is an a.s.\ unique quantile process $Q^\mu$ with $Q^\mu_{\#}(\lambda^N)=\mu$.  
\end{lemma} 

In essence this follows by disintegrating $\mu$ into a product of successive kernels
\begin{align}\label{eq:kernel1}
\mu(du_1, \ldots, du_N)= \mu_1(du_1)\mu_2^{u_1}(du_1) \ldots \mu_{N}^{u_1...u_{N-1}}(du_N)
\end{align}
and taking the respective quantile functions as components of $Q^\mu$. 
 We call $\mu$ \emph{triangular regular} if all $\mu_k^{u_1...u_{k-1}}, k\leq N$, are atom free. 
 In this case the consistency condition \eqref{CanonicalIncreasingII} can be omitted. 
Every absolutely continuous measure is triangular regular.
We refer to \Cref{lem:Q.defs} for details.

\subsection{Knothe--Rosenblatt distance and adapted weak topology}

\begin{definition}\label{def:kr} For   $\mu, \nu\in \mathcal P_p(\R^N)$ 
    the \emph{Knothe--Rosenblatt coupling} is  given by  
\begin{align}
 \kr_{\mu,\nu}:=(Q^\mu, Q^\nu)_{\#}(\lambda^N).
\end{align}
For  $p\in \{0\} \cup [1,\infty)$ the $p$-\emph{Knothe--Rosenblatt distance} is defined by 
\begin{align}\label{KRDDef} 
\KR_p^p(\mu,\nu):= & \textstyle  \int \|x-y\|^p_p\, d \kr_{\mu,\nu}(x,y)= \int |Q^\mu - Q^\nu|_p^p\, d\lambda^N,
\end{align}
where we abuse notation for $p=0$ and convene that  $\|x-y\|^0_0:= |x-y|\wedge 1$.  
\end{definition}

We formulate our first contribution:
\begin{theorem}\label{thm:Main}
   For $p\in \{0\}\cup [1, \infty)$ the  Knothe--Rosenblatt distance $\KR_p$ metrizes the ($p$-) adapted weak topology on $\P_p(\R^N)$. 
\end{theorem}

The \emph{adapted weak topology} refines the usual weak topology on $\P_p(\R^N)$ with the purpose of taking the flow of information into account. It has been introduced independently by different groups of authors: by Aldous \cite{Al81} based on the prediction process, by Hellwig \cite{He96} using conditional distributions, by Hoover--Keisler \cite{HoKe84} using adapted functions, and by Pflug--Pichler \cite{PfPi12, PfPi14}, Rüschendorf \cite{Ru85}, Lasalle \cite{La18}, and Nielson--Sun \cite{NiSu20} using adapted variants of the Wasserstein distance. Importantly all of these constructions yield the same topology in finite discrete time (\cite{BaBaBeEd19b}).
We refer to \cite{Pa22} for a recent survey and will give a precise definition of the adapted Wasserstein distance in \Cref{ssec:AW} below, when we compare the Knothe--Rosenblatt distance and the adapted Wasserstein distance.

\begin{remark}
It is well known (and not hard to see) that convergence in $p$-Wasserstein distance is equivalent to weak convergence plus convergence of $p$-th moments. Likewise 
convergence in adapted $p$-Wasserstein distance is equivalent to adapted weak convergence plus convergence of $p$-th moments (see e.g.\ \cite{BaBaBeEd19b}).  
\end{remark}

\subsection{Isometric embedding and adapted Skorokhod representation theorem}

It follows from \eqref{KRDDef}  that
\begin{align}\label{eq:embedding} Q: (\P_p(\R^N), \KR_p) \to (L^p(\lambda^N), \|. \|_p) , \quad   \mu \mapsto Q^ \mu
\end{align} 
is an isometric embedding for $p\in \{0\}\cup [1, \infty)$. We can thus rephrase Theorem \ref{thm:Main} as follows: 
\begin{corollary}[adapted Skorokhod representation theorem] \label{cor:Skorokhod}
    Let $p\in \{0\} \cup [1,\infty)$ and $\mu_n, \mu \in \P_p (\R^N), n \in \N$. Then $\mu_n \to \mu$  in the $p$-adapted weak topology if and only if  $Q^{\mu_n} \to Q^\mu$ converges  in $\|.\|_p$.
\end{corollary}

Recall that the Skorokhod representation theorem asserts that for a sequence of probabilities $\mu_n, \mu \in \mathcal P_p(\R), n \in \N,$ on the real line we have $\mu_n \to \mu$ weakly if and only if on some probability space $(\Omega, \P)$ there are random variables $X_n\sim \mu_n, X\sim\mu, n \in \N$ with $X_n \to X$ in probability.
Indeed, in this case one can take $(\Omega, \P):=((0,1), \lambda)$, $X_n:= Q^{\mu_n}$ and $X := Q^\mu$.

In complete analogy, Corollary \ref{cor:Skorokhod} asserts that adapted weak convergence of a sequence of probabilities $\mu_n \to \mu$ is equivalent to the convergence of the respective quantile processes in probability on $((0,1)^N, \lambda^N)$. 
We note that in contrast to the classical Skorokhod theorem, the assertion on almost sure convergence does not hold in the multi-period setting, as we will discuss in \Cref{cor:Bogachev} below.

\subsection{Completion of $(\P_p(\R^N), \KR_p)$}

Equipped with the Knothe--Rosenblatt distance $\KR_p$, the space $\P_p  (\R^N)$ is not complete.\footnote{Also all the other known natural metrics of the weak adapted topology are not complete. In \cite{BaBePa21} it is established that the completion of $\P_p(\R^N)$ w.r.t.\ adapted Wasserstein distance can be identified with the set of filtered stochastic processes.} 
In view of the isometric embedding \eqref{eq:embedding} it is clear that the completion  $\overline{\P_p  (\R^N)}^{\KR_p}$ of $\P_p(\R^N)$ consists precisely in the  $\|.\|_p$-closure of all quantile processes with finite $p$-th moment. In the next result we characterize this closure. Following \cite{BoKoMe05} we call a triangular transformation $T=(T_k)_{k=1}^N:(0,1)^N \to \R$ \emph{increasing} if
it satisfies \eqref{CanonicalIncreasingI}, i.e.\ for all $k\leq N, u_1, \ldots, u_k, u_k'$  with $u_k \leq u'_k$
\begin{align*} 
 T_k(u_1, \ldots,  u_k) &\quad \leq \quad  T_k(u_1, \ldots, u_k'). 
\end{align*}
We call  $T=(T_k)_{k=1}^N:(0,1)^N \to \R$ \emph{strictly increasing} if
it  satisfies  for all $k\leq N, u_1, \ldots, u_k, u_k'$  with $u_k < u'_k$
\begin{align}\label{eq:sinc} \tag{\text{sinc}}
 T_k(u_1, \ldots,  u_k) &\quad < \quad  T_k(u_1, \ldots, u_k'). 
\end{align}
We write $L^\text{{\rm IT}}_p(\lambda^N)$ / $L^\text{{\rm SIT}}_p(\lambda^N)$ for the set of all increasing / strictly increasing triangular transformations which are $p$-integrable and we write $L^{{\rm quantile}}_p(\lambda^N)$ for the set of $p$-integrable quantile processes. 
\begin{theorem}
    \label{thm:closure} Let $p\in \{0\} \cup [1,\infty)$, then $L_p^{\rm SIT}(\lambda^N)$, $L_p^{\rm quantile}(\lambda^N)$ and $L_p^{\rm IT}(\lambda^N)$ are convex cones with:
    \begin{enumerate}
    \item The $\|.\|_p$-closure of all $p$-integrable quantile processes $L^{\rm quantile}_p(\lambda^N)$ consists precisely in $L^\text{{\rm IT}}_p(\lambda^N)$. In particular  the completion $ \overline{\P_p  (\R^N)}^{\KR_p}$ of $\P_p(\R^N)$ is  $L^\text{{\rm IT}}_p(\lambda^N)$ up to isometry.
    \item $\mu \in \P_p(\R^n)$ is triangular regular if and only if $Q^\mu\in L^\text{{\rm SIT}}_p(\lambda^N)$.
    \item $L^\text{{\rm SIT}}_p(\lambda^N)$ is a dense subset of $L^{{\rm quantile}}_p(\lambda^N)$.
 \end{enumerate}
\end{theorem}
The following diagram summarizes the relations between these sets:
   \[
    \begin{array}{ccc}
        L^\text{{\rm SIT}}_p(\lambda^N) & = & \{Q_\mu: \mu \in \P_p(\R^N), \mu \text{ triangular regular}\} \\
   &     \rotatesupseteq  & \\
        L^{{\rm quantile}}_p(\lambda^N) & = & \{Q_\mu: \mu \in \P_p(\R^N)\} \\
      &  \rotatesupseteq  & \\
        L^\text{{\rm IT}}_p(\lambda^N) & = & \overline{\{Q_\mu: \mu \in \P_p(\R^N)\}}
    \end{array}.
\]

\subsection{Geodesic completeness and barycenters 
}
Starting with the work of McCann \cite{Mc91} the fact that Wasserstein distance on $\mathcal P_p(\R^N)$ is \emph{geodesic} has played a tremendous role in optimal transport and its applications. 
In the case of the Knothe--Rosenblatt distance, the geodesic structure becomes very simple due to the isometry \eqref{eq:embedding} which maps into the linear space $L^p(\lambda^N)$. 
We have seen in Theorem \ref{thm:closure} that the set of quantile processes is a convex cone and thus obtain:
\begin{corollary}\label{cor:GeoComplete}
The sets $\P_p(\R^N)$, $ \{ \mu \in \P_p(\R^N):  \mu \text{ triangular regular}\}$, $\overline{\P_p  (\R^N)}^{\KR_p}$  are geodesically complete w.r.t.\ $\KR_p$ for $p\geq 1$ and geodesics are unique for $p>1$.
\end{corollary}

The notion of Wasserstein barycenter was famously introduced by Agueh and Carlier \cite{AgCa11}. Naturally we can define the analogous concept for the Knothe--Rosenblatt distance. Using again the isometry \eqref{eq:embedding} it follows that for  $\mu_1, \ldots, \mu_n \in \mathcal P_p(\R^N)$ and convex weights $\lambda_1, \ldots, \lambda_n$ there exists a unique (for $p>1$) \emph{Knothe--Rosenblatt barycenter} 
i.e.\ a probability
$\bar\mu\in \mathcal P_p(\R^N)$ that minimizes  
$$\inf_{\mu\in \mathcal P_p(\R^N)}  \lambda_1 \KR_p^p(\mu, \mu_1) + \ldots + \lambda_n \KR_p^p(\mu, \mu_n),$$
indeed $\bar \mu$ is given by $Q^{\bar \mu}= \lambda_1 Q^{\mu_1} +  \ldots + \lambda_n Q^{\mu_n}$.

We also establish that the geometric structure of $(\P_p(\R^N), \KR_p)$ is compatible with properties that are interesting from a stochastic process perspective. In particular we show that the set of martingale measures and the set of all probabilities that correspond to predictable processes are geodesically convex, see \Cref{ssec:preservative}.

\subsection{Bicausal couplings and adapted Wasserstein distance} \label{ssec:AW}

Adapted versions of the classical Wasserstein distance have been independently introduced by different authors, 
see  
 \cite{Ru85, Gi04, PfPi12, PfPi14, BiTa19, NiSu20}. Here  `adaptedness' is incorporated in its definition by imposing an additional causality constraint on the set of admissible  couplings that we introduce next:

Given $\mu, \nu \in \P_0 (\R^N)$ the set $\cpl(\mu,\nu)$ consists of all measures $\pi \in \P_0 (\R^N\times \R^N)$ with marginals $\mu$ and $\nu$. 
The set $\cplbc(\mu, \nu)$ of \emph{bicausal couplings} consists of those $\pi$ whose successive disintegration kernels satisfy a.s.\
$$\pi_k^{x_1, \ldots, x_{k-1}, y_1, \ldots, y_{k-1}} \in \cpl(\mu_k^{x_1, \ldots, x_{k-1}},  \nu_k^{y_1, \ldots, y_{k-1}}), \quad  k \leq N,$$
where we disintegrate $\pi$ ($\mu$ and $\nu$) analogous to \eqref{eq:kernel1} above, i.e.\
\[
    \pi(dx_1, dy_1, \ldots, dx_N, dy_N) = \pi_1 (dx_1, dy_1) \pi_{2}^{x_1, y_1} (dx_2, dy_2)
    \ldots \pi_{N-1}^{x_1, y_1 \ldots, x_{N}, y_{N-1}} (dx_N, dy_N).
\]
If $\mu, \nu$ are both triangular regular, the set of bicausal couplings is the closure of Monge couplings which are concentrated on the graph of a triangular mapping with triangular inverse, see \cite{BePaSc22}.

The adapted Wasserstein distance of $\mu, \nu\in \mathcal P_p(\R^N)$ is given by 
\begin{align*}
     \AW_p^p(\mu, \nu)& := \inf\{\textstyle \int |x-y|_p^p\, d\pi: \pi \in \cplbc(\mu, \nu)\}.
\end{align*}
A frequently used definition of the Knothe--Rosenblatt coupling is the following:
 $\pi$ is the Knothe--Rosenblatt coupling between $\mu$ and $\nu$ if and only if we have $\pi$-almost surely that for $k \leq N$ and $x_1, \ldots, x_{k-1}, y_1, \ldots, y_{k-1}$ 
\begin{align}\label{eq:KRDisDef}
\pi_k^{x_1, \ldots, x_{k-1}, y_1, \ldots, y_{k-1}} \text{  is the quantile coupling of } \mu_k^{x_1, \ldots, x_{k-1}} \text{ and } \nu_k^{y_1, \ldots, y_{k-1}},   
\end{align}
see e.g.\ \cite{BaBeLiZa17}. In particular, the Knothe--Rosenblatt coupling is bicausal and thus 
  $\AW_p(\mu, \nu) \leq \KR_p(\mu, \nu)$. 
We collect a number of further results comparing $\AW_p$ and $\KR_p$
\begin{enumerate}
    \item $\AW_p$ and $\KR_p$ are not equivalent as metrics. Indeed their completions are \emph{incomparable} as we discuss in Section \ref{sec:DifferentComp} below. 
    \item $\AW_p$ and $\KR_p$ \emph{are} equivalent metrics on the set of triangular $L$-Lipschitz measures, see Theorem  \eqref{thm:EquivalentMetrics} below. Here we say that $\mu$ is \emph{triangular $L$-Lipschitz} if the map $$\R^{k-1}\ni (x_1, \ldots, x_{k-1}) \to \mu_k^{x_1, \ldots, x_{k-1}} \in \P_p(\R)$$
    is $L$-Lipschitz w.r.t.\ the $p$-Wasserstein distance on $\P_p(\R)$  for each $k\leq n$.
    \item In \cite[Corollary 2]{Ru85} Rüschendorf shows that under a \emph{monotone regression assumption}
    $\AW_p(\mu, \nu) = \KR_p(\mu, \nu)$, see also \cite[Proposition 5.3]{BaBeLiZa17} and \cite[Proposition 3.5]{BaKaRo22}.
\end{enumerate}

It is shown in \cite{BaBePa21} that the  completion of $\mathcal P_p(\R^N)$ w.r.t.\ the adapted Wasserstein distance can be interpreted as the set $\mbox{FP}_p$ of all \emph{stochastic processes with filtration}. 
While $(\mbox{FP}_p, \AW_p)$ is geodesically complete,  $\mathcal P_p(\R^N)$ itself is not geodesically complete when equipped with the adapted Wasserstein distance. 

We note that  in contrast to the Knothe--Rosenblatt case, $\AW_p$-compact subsets of $\mbox{FP}_p$ admit a simple characterization  in the spirit of Prokhorov's theorem, see \cite[Theorem 1.7]{BaBePa21}. 
We consider this as a distinct advantage of $\AW_p$ over  $\KR_p$.

\subsection{Knothe--Rosenblatt Rearrangement in $(\R^d)^N$} \label{ssec:generalization}

In view of applications, it is important to consider the adapted weak topology on $\P_p((\R^d)^N)$, where $d\in \N$ stands for the dimension of the \emph{state space}, while $N\in \N$ denotes the number of time steps. 
Given that the Knothe--Rosenblatt distance is a particularly simple metric for the adapted weak topology, it is interesting to ask for extensions to $d>1$.  
We are aware of two obvious generalizations that are suitable for our purposes. 

The more interesting (from our perspective) extension of the Knothe--Rosenblatt coupling to multiple dimensions is based on defining the quantile process for the state space $\R^d$. 
In the case $p=2$ this amounts to: 
\begin{definition} \label{def:quantile_process_multi}
A triangular $Q=(Q_k)_{k=1}^N: ((0,1)^d)^N\to (\R^d)^N$  is a \emph{quantile process} if for all $k\leq m \leq N, u_1, \ldots, u_m, u_k' \in (0,1) $ with $u_k \leq u'_k$ 
\begin{align}\label{CanonicalIncreasingIMulti2} 
 Q_k(u_1, \ldots,  u_{k-1}, \cdot) & \quad  \text{ is the gradient of a convex function } \tag{\text{$|.|_2$-mon}} \\
 Q_k(u_1, \ldots, u_k) = Q_k(u_1, \ldots, u_k')  &  \Rightarrow Q_m(u_1, \ldots, u_k, \ldots, u_m )= Q_m(u_1, \ldots, u_k', \ldots, u_m). \tag{\text{con}}  \label{CanonicalIncreasingIIMulti2}
\end{align}
\end{definition}
Condition \eqref{CanonicalIncreasingIMulti2} asserts that $Q_k(u_1, \ldots,  u_{k-1}, \cdot)$ is a $\W_2$-optimal transport map. For general $p>0$ we would consider $\W_p$-optimal mappings, see    \cite[Theorem 3.7]{GaMc96}. The monotonicity condition \eqref{CanonicalIncreasingIMulti2} can then be rephrased (for instance) as  
\begin{align}\label{CanonicalIncreasingIMultiP}
 Q_k(u_1, \ldots,  u_{k-1}, \cdot) & \quad  \text{ is $|\cdot - \cdot|_p^p$-cyclically monotone}. \tag{\text{$|.|_p$-mon}}
\end{align}
As above, for each $\mu\in \P_p((\R^d)^N)$ there exists a unique quantile process $Q^\mu$ that pushes $(\lambda^d)^N$ to $\mu$.
Analogously to the one-dimensional case we define a generalized Knothe--Rosenblatt coupling and distance through
\[
    \kr_{\mu,\nu} := \text{law}( Q^\mu, Q^\nu) \quad\text{and}\quad
    \KR_p^p(\mu,\nu) := \int |Q^\mu - Q^\nu |_p^p \, d\lambda^{dN}.
\]
For this construction, our main result extends:
\begin{theorem}
    \label{thm:KRandAW}
    Let $p > 1$. Then the map
    \[
        (\P_p((\R^d)^N), \KR_p) \to (L_p((\lambda^d)^N),\|.\|_p) \colon \mu \to Q^\mu
    \]
    is an isometry.
    $\KR_p$ metrizes the $p$-adapted weak topology. That is, for $\mu, \mu_n\in \P_p((\R^d)^N), n \in \N$ we have 
    \[
        Q^{\mu_n} \to Q^\mu\text{ in $L_p((\lambda^d)N)$} \iff \KR_p(\mu_n, \mu) \to 0 \iff \AW_p(\mu_n, \mu) \to 0.
    \]
\end{theorem}

Another possibility is to  interpret the definition of the Knothe--Rosenblatt distance given in \eqref{eq:KRDisDef} in multiple dimensions.  Fix $p\geq 1$ and $\mu, \nu\in \P_p((\R^d)^N)$.
We say that $\pi\in \cpl(\mu, \nu)$ is \emph{triangular optimal} 
if we have $\pi$-almost surely that for $k \leq N$ and $x_1, \ldots, x_{k-1}, y_1, \ldots, y_{k-1}$ 
\begin{align}\label{eq:KRDisDef.multidim}
\pi_k^{x_1, \ldots, x_{k-1}, y_1, \ldots, y_{k-1}} \text{ is a $\W_p$-optimal coupling between } \mu_k^{x_1, \ldots, x_{k-1}} \text{ and } \nu_k^{y_1, \ldots, y_{k-1}}.
\end{align}
We denote the set of all triangular optimal couplings by $\cpl_{\text{to}}(\mu, \nu)$ and set
\begin{align}\label{eq:KRMultiTilde}
\widetilde \KR_p^p(\mu, \nu):= \sup_{\pi\in \cpl_{\text{to}}(\mu, \nu)} \int |x-y|^p\, d\pi(x,y).
\end{align}
With these definitions we have: 
\begin{theorem}\label{thm:tildeKRandAW} 
Let $p\geq 1$. Then $\widetilde \KR_p$ induces the $p$-weak adapted topology. That is,  for $\mu, \mu_n\in \P_p((\R^d)^N), n\in \N$ we have $\widetilde\KR_p(\mu_n, \mu) \to 0 $ if and only if $\AW_p(\mu_n, \mu) \to 0$.
\end{theorem}
However, $\widetilde \KR_p$  does \emph{not} satisfy the triangular inequality and is thus not a metric. (See section \ref{sec:OnWideTilde}.)

\subsection{Literature}

The Knothe--Rosenblatt coupling was independently introduced by Knothe \cite{Kn57} and Rosenblatt \cite{Ro52} for applications to geometric inequalities.

The notion of triangular transformation was introduced in  \cite{BoKoMe05} and applied for the derivation of (generalized) Talgrand- and logarithmic Sobolev inequalities.  Among further results, this article also establishes that convergence in total variation implies convergence of the corresponding increasing triangular transformations, see also Corollary \ref{cor:Bogachev} below. 
Triangular transformation have seen increasing use in machine learning.
E.g.\ triangular transformations play a central role in \cite{PaMa18} where a Metropolis-Hastings algorithm for sampling complex high-dimensional distributions is proposed.
In \cite{SpBiMa18} sparse triangular transport maps are used  as a tractable class of transformation to tackle filtering and variational inference problems.
They are employed for conditional density estimation and structure learning in \cite{BaMaZa20}.
In \cite{XuLiMuAc20} neural networks with triangular structure are used to parametrize the dual variables of the causal transport problem.

We emphasize that following the contributions on quantile / increasing triangular transformations in  \cite{Ru85, BoKoMe05, BaBeLiZa17} (among others) the results in the preparatory Section \ref{sec:preparations} below are folklore.

Continuous-time versions of the Knothe--Rosenblatt coupling are presented in \cite{BiTa19, BaKaRo22}. In particular these articles extend results of \cite{Ru85} on the optimality of the Knothe--Rosenblatt distance in the bi-causal transport problem to continuous time and consider the induced topology on laws of stochastic processes.

The article  \cite{CaGaSa09} connects the Knothe--Rosenblatt coupling to classical optimal transport: for triangular regular probabilities, the Knothe--Rosenblatt transformation is the limit of Brenier maps when assigning progressively smaller weights to later coordinates.

Adapted versions of the Wasserstein distance play an important role in stochastic optimization and multistage programming, see e.g.\  \cite{Pf09, PfPi12, KiPfPi20, BaWi23}. Adapted weak topologies are a useful tool for various problems in mathematical finance, e.g.\ in the pricing of game options \cite{Do14}, questions of insider trading and enlargement of filtrations \cite{AcBaZa20},  stability of pricing / hedging and utility maximization  \cite{GlPfPi17, BaDoDo20, BaBaBeEd19a, BeJoMaPa21a} and  interest rate uncertainty \cite{AcBePa20}. We refer to \cite[Sections 1.8, 1.9]{BaBePa19} for a more complete account on the literature on adapted Wasserstein distance and adapted weak topologies.

\subsection{Organisation of the paper}
In Section \ref{sec:preparations} we introduce notation and basic results. Specifically, we discuss the representation of the Knothe--Rosenblatt coupling through quantile processes (Lemma \ref{lem:Qmu}) and the alternative characterization of the Knothe--Rosenblatt coupling given in \eqref{eq:KRDisDef} above. 

In Section \ref{sec:ProofMainResult} we give the proof of our main result that the Knothe--Rosenblatt distance induces the weak adapted topology based on the notion of modulus of continuity of measures introduced in \cite{Ed19}. As a consequence of our argument we also obtain that $\AW_p$ and $\KR_p$ are equivalent for (uniformly) triangular Lipschitz kernels. 

In Section \ref{sec:IsometryConsequences} we draw consequences from the isometry \eqref{eq:embedding} between the Knothe--Rosenblatt distance and the $p$-norm on $L_p(\lambda ^N)$. In particular we prove the results concerning the metric completion of $\KR_p$ as well 
as on geodesic completeness announced above.

Finally the appendix is concerned with the Knothe--Rosenblatt distance in the case of a multiple dimensional state space.

\section{Preparations}\label{sec:preparations}

\subsection{Setting and Notation}

For a Polish metric space $(\mathcal X, d_\mathcal X)$, we denote by $\mathcal P_p(\mathcal X)$ the set of Borel probability measures on $\mathcal X$ that finitely integrate $d_\mathcal X(x,x_0)^p$ for some (thus any) $x_0 \in \mathcal X$.
For the special case $p = 0$, we use $\mathcal P_0$ and $\mathcal P$ synonymously.
When $T \colon \mathcal X \to \mathcal Y$ is a measurable map between Polish spaces and $\mu \in \mathcal P(\mathcal X)$, we write $T_\# \mu$ for the push-forward measure of $\mu$ under $T$.
The product space $\mathcal X \times \mathcal Y$ is endowed with the product topology and, in case of Polish metric spaces, with the product metric $d_p((x_0,y_0),(x_1,y_1))^p = d_\mathcal X(x_0,x_1)^p + d_\mathcal Y(y_0,y_1)^p$ where $x_0,x_1 \in \mathcal X$ and $y_0, y_1 \in \mathcal Y$.
A probability $\pi \in \mathcal P(\mathcal X \times \mathcal Y)$ is called coupling with marginals $(\mu,\nu) \in \mathcal P(\mathcal X) \times \mathcal P(\mathcal Y)$ 
if $\proj^1_\# \pi = \mu$ and $\proj^2_\# \pi = \nu$ where we use $\proj^i$ to denote the $i$-th coordinate projection.
A coupling $\pi$ is said to be concentrated on the graph of a measurable function, if there exists a measurable map $T \colon \mathcal X \to \mathcal Y$ such that $(\id,T)_\# \mu = \pi$. 
Given a probability $\mu \in \mathcal P(\mathcal X)$ we write $L_p(\mu;\mathcal Y)$ for the set of $p$-integrable functions (when $p \ge 1$) resp.\ measurable functions (when $p = 0$) that take values in $\mathcal Y$.
We equip $L_p(\mu;\mathcal Y)$ with the usual $L_p$-distance whereas we understand under $L_0$-convergence simply convergence in $\mu$-probability.
The $p$-Wasserstein distance between $\mu,\nu \in \mathcal P_p(\mathcal X)$ is given by
\[
    \mathcal W_p(\mu,\nu) :=
    \begin{cases}
        \inf_{\pi \in \Pi(\mu,\nu)} \left( \int d_\mathcal X(x,y)^p \, d\pi(x,y) \right)^\frac1p & p \in [1,\infty), \\
        \inf_{\pi \in \Pi(\mu,\nu)} \int d_\mathcal X(x,y) \wedge 1 \, d\pi(x,y) & p = 0,
    \end{cases}
\]
and we recall the well-known fact that $\mathcal W_0$ metrizes the topology of weak convergence on $\mathcal P(\mathcal X)$.

Let $\mu,\nu \in \mathcal P(\prod_{k = 1}^N \mathcal X_k)$ where $(\mathcal X_k,d_{\mathcal X_k})$ are Polish metric spaces.
For $\ell, k \in \mathbb N$, $\ell, k \le N$, we will write $x_{\ell:k}$ to denote the vector $x_\ell,\ldots,x_k$ and similarly write $\mathcal X_{\ell:k}$ for the product $\mathcal X_\ell \times \dots \times \mathcal X_k$.

In the arguments below it will be important to keep track of properties of successive disintegrations of probability measures such as in \eqref{eq:kernel1}. 
To this end we denote by $$K^\mu_k \colon \mathcal X_{1:k-1} \to \mathcal{P}(\mathcal X_k)$$ 
a regular disintegration of $\proj^{1:k}_\# \mu$ w.r.t.\ the first $k-1$-coordinates and remark that $K^\mu_k$ is Borel measurable.
Further we use the convention that $K^\mu_1(x_{1:0};dx_1) = \proj^1_\# \mu(dx_1)$.
With this notation in hand we can rewrite \eqref{eq:kernel1} as
\[
    \mu(dx_{1:N}) = \bigotimes_{k = 1}^N K^\mu_k(x_{1:k-1};dx_k).
\]
 Recall that $\mu \in \mathcal{P}(\R^N)$ {triangular regular} if $\mu$-almost surely, for $k = 1,\ldots,N$, $K^\mu_k$ takes values in the set of regular measures, that is, $K^\mu_k(x_{1:k-1}) \in \{ \rho \in \mathcal P(\R) : \rho \text{ is non-atomic} \}$ for $\mu$-a.e.\ $x$. 
A coupling $\pi \in \Pi(\mu,\nu)$ is bicausal if for all $k \le N$, the disintegration satisfies for$((x_\ell,y_{\ell})_{\ell = 1}^N) \in \prod_{k = 1}^N \mathcal X_k \times \mathcal X_k$ $\pi$-a.s.\
\[
    K^{\pi}_k\left((x_\ell,y_\ell)_{\ell = 1}^{k - 1}\right) 
    \in 
    \Pi\left(K^\mu_k(x_{1:k-1}), K^\nu_{k}(y_{1:k-1} \right).
\]
The set of bi-causal transport plans in $\Pi(\mu,\nu)$ is compact, see e.g.\ \cite{BaBeLiZa17}. 
The $p$-adapted Wasserstein distance between $\mu$ and $\nu$ is given by
\begin{equation}
    \label{eq:def.AW}
    \AW_p(\mu,\nu) :=
    \begin{cases}
        \inf_{\pi \in \Pi_{\rm bc}(\mu,\nu)} \left( \int d_{\mathcal X}(x,y)^p \, d\pi(x,y) \right)^\frac1p & p \in [1,\infty), \\
        \inf_{\pi \in \Pi_{\rm bc}(\mu,\nu)} \int d_{\mathcal X}(x,y) \wedge 1 \, d\pi(x,y) & p = 0.
    \end{cases}
\end{equation}
We recall that $\AW_0$ metrizes the adapted weak topology on $\mathcal P(\mathcal X_{1:N})$, see, e.g.\ \cite{BaBaBeEd19b, Pa22}.

The gluing $\mu \dot \otimes \nu$ of measures $\mu \in \mathcal P(\mathcal A \times \mathcal B)$ and $\nu \in \mathcal P(\mathcal B \times \mathcal C)$ that share a common marginal $\proj^2 \mu = \proj^1 \nu$, is defined by
\[
    \mu K^{\nu}_1 (A \times C) = \int_{A \times \mathcal B} K^\nu_1(b;C) \, d\mu(a,b)\quad  \text{ for all measurable $A\subseteq \mathcal A, B\subseteq \mathcal B$}.
\]

\subsection{The quantile process}
The next lemma establishes the correspondence between quantile processes and probabilities on $\R^N$. 

\begin{lemma}\label{lem:Q.defs}
    Let $\mu \in \mathcal P(\mathbb R^N)$. Then we have:
    \begin{enumerate}[label = (\alph*)]
        \item The process $Q^\mu$ given by $Q^\mu_k(u_{1:k}) := Q^{K^\mu_{k - 1}(Q^\mu_{1:k-1}(u_{1:k-1}))}(u_k)$ satisfies \eqref{CanonicalIncreasingI} and \eqref{CanonicalIncreasingII}.
        \item If $Q\in L_0(\lambda^N;\mathbb R^N)$ satisfies \eqref{CanonicalIncreasingI} and \eqref{CanonicalIncreasingII}, then $Q^\mu = Q$ for $\mu = Q_\# \lambda^N$.
    \end{enumerate}
\end{lemma}

In the proof of Lemma \ref{lem:Q.defs} we need to derive certain measurability properties from the consistency condition \ref{CanonicalIncreasingII}.
To this end, we first establish an auxiliary lemma that provides an adequate connection.
Even though the assertion of Lemma \ref{lem:consistency.measurability} is trivial for discrete $\mu$, we have to use some machinery of descriptive set theory in the general case.

\begin{lemma}
    \label{lem:consistency.measurability}
    Let $\mu \in \mathcal P(\mathcal X)$, $f,g \colon \mathcal X \to \mathcal Y$ be measurable such that for all $x_1,x_2 \in \mathcal X$, $f(x_1) = f(x_2)$ implies $g(x_1) = g(x_2)$.
    Then there exists a Borel measurable function $h \colon \mathcal Y \to \mathcal X$ such that
    \[
        g = g\circ h \circ f \quad \mu\text{-almost surely}.
    \]
\end{lemma}

\begin{proof}
    Since $\mathcal X$ is Polish and $f$ is measurable, the graph of $f$ is also Borel measurable.
    By the Jankow-von Neumann uniformization theorem there is an analytically measurable function $\tilde h \colon \mathcal Y \to \mathcal X$ with $f = f \circ \tilde h \circ f$.
    Using our assumption on $g$, we get that $g = g \circ \tilde h \circ f$.
    Replacing $\tilde h$ by a Borel measurable function such that $f_\#(\mu)$-almost surely $\tilde h = h$, we have found a function with the desired property.
\end{proof}

\begin{proof}[Proof of Lemma \ref{lem:Q.defs}]
    The first assertion is obvious from the definition of $Q^\mu$.

    If $Q$ satisfies \eqref{CanonicalIncreasingI} and \eqref{CanonicalIncreasingII}, then $\pi := (\id,Q)_\# \lambda^N \in \cplbc(\lambda^N,\mu)$.
    Indeed,  let $U$ be uniformly distributed on $(0,1)^N$, $X = Q(U)$, and fix $k \in \{1,\ldots,N\}$.
    By Lemma \ref{lem:consistency.measurability} there is a measurable function $h \colon \mathbb R^k \to (0,1)^k$ such that almost surely $X_{1:k-1} = Q_{1:k-1} \circ h (X_{1:k-1})$.   
    Since $Q$ is consistent we have
    \[
        X_{k} = Q_{k}(U) = Q_{k}(h(X_{1:k-1}),U_{k}),
    \]
    whence $X_{k}$ is $(X_{1:k-1},U_{k})$-measurable.
    Since $U_{1:k-1}$ is independent of $U_{k}$, this proves that
    \[
        X_{k} \text{ is conditionally independent of }U_{1:k-1} \text{ given }X_{1:k-1}.
    \]
    On the other hand, since $X_{1:k-1}$ is $U_{1:k-1}$-measurable, we have independence of $U_{k}$ and $(U_{1:k-1},X_{1:k-1})$.
    To summarize, using these conditional independence properties we have shown that almost surely
    \begin{align}\label{eq:GudiStar}
        K_k^\pi((U_\ell, X_\ell)_{\ell = 1}^{k-1}) = Law( U_{k}, X_{k} |  (U_\ell, X_\ell)_{\ell = 1}^{k - 1 }) \in \cpl(\lambda, K^\mu_k((X_\ell)_{\ell = 1}^{k - 1 })).
    \end{align}
    We conclude that $\pi$ is bicausal.
    
    It was established in \cite[Proposition 5.9]{BaBeLiZa17} that $Q^\mu$ is the only triangular increasing map that induces a coupling in $\cplbc(\lambda^N,\mu)$, hence, $Q = Q^\mu$.
\end{proof}

As direct corollary of Lemma \ref{lem:Q.defs} we obtain Lemma \ref{lem:Qmu}.
From now on, we will denote by $Q^\mu$ the map $Q \in L_0(\lambda^N;\mathbb R^N)$ that is uniquely determined by \eqref{CanonicalIncreasingI}, \eqref{CanonicalIncreasingII}, and $Q_\# (\lambda^N) = \mu$.

\begin{proposition}
    Let $\mu \in \mathcal P(\mathbb R^N)$ and $Q \in L_0(\lambda^N;\mathbb R^N)$ with $Q_\# (\lambda^N) = \mu$.
    \begin{enumerate}[label = (\alph*)]
        \item $Q = Q^\mu$ iff $Q$ is triangular increasing and $(\id,Q)_\# \lambda^N \in \cplbc(\lambda^N,\mu)$.
    \end{enumerate}
    Additionally assume that $\mu$ is triangular regular, then
    \begin{enumerate}[label = (\alph*), resume]
        \item $Q = Q^\mu$ iff $Q$ is triangular increasing.
    \end{enumerate}
\end{proposition}

\begin{proof}
    The first assertion was shown in \cite[Proposition 5.9]{BaBeLiZa17}. 
    To see the second assertion,
    recall that when $\rho \in \mathcal P(\mathbb R)$ is regular, then the quantile function $Q^\rho$ is the unique increasing map that pushes $\lambda$ to $\rho$.
    When $\mu$ is triangular regular, then its disintegration $K^\mu_k$ consists $\mu$-almost surely of regular measures.
    Assume that $Q$ is triangular increasing.
    Then we have for $k = 1$ that $Q_1$ pushes $\lambda$ to $K^\mu_1$, thus,  
    $Q_1 = Q^{K^\mu_1} = Q^\mu_1$.
    Given $Q_{1:k} = Q^\mu_{1:k}$ we can repeatedly apply this reasoning in an induction to obtain $Q_{k + 1} = Q^\mu_{k + 1}$.
\end{proof}

We recall from Definition \ref{def:kr} that the Knothe--Rosenblatt coupling of two distributions $\mu,\nu \in \mathcal P(\mathbb R^N)$ is given by $\kr_{\mu,\nu} = Law(Q^\mu,Q^\nu)$.

\begin{lemma} \label{lem:KR.bicausal}
    Let $\mu, \nu \in \mathcal P(\mathbb R^N)$ and $\pi \in \cpl(\mu, \nu)$. Then $\pi$ is the Knothe--Rosenblatt coupling of $\mu$ and $\nu$ if and only if $\pi$ satisfies \eqref{eq:KRDisDef}.
    In particular, $\W_p \le \AW_p \le \KR_p$.
\end{lemma}

\begin{proof}
    It follows from \eqref{eq:KRDisDef} and the definition of the quantile coupling on the real line that $\kr_{\mu,\nu}$ satisfies \eqref{eq:KRDisDef}.
    If $\pi$ satisfies \eqref{eq:KRDisDef}, then the disintegrations $K^\pi_k$ and $K^{\kr_{\mu,\nu}}_k$ coincide $\pi$-almost surely for $k = 1,\ldots,N$.
    As the disintegration uniquely determines a coupling, we conclude that $\pi = \kr_{\mu,\nu}$.
\end{proof}

\section{The Knothe--Rosenblatt distance induces the adapted weak topology}\label{sec:ProofMainResult}

We have seen in Lemma \ref{lem:KR.bicausal} that the $\AW_p$-topology is coarser than the $\KR_p$-topology.
The main goal of this section is to establish their equivalence on $\mathcal P_p(\mathbb R^N)$.
The proof that we give in this section, directly makes use of the concept of modulus  of continuity for measures introduced by Eder in \cite{Ed19} and quantifies how much $\AW_p$ and $\KR_p$ differ from each other.
One step in this argument (in equation \eqref{eq:direct.proof.remark} below) utilizes the 1-dimensional fact that the composition of monotone couplings is again monotone, and thus optimal for Wasserstein distances.
This property breaks down for the generalized Knothe--Rosenblatt distance $\KR_p$ on $\mathcal P_p((\R^d)^N), d \geq 2$, introduced in Subsection \ref{ssec:generalization}.
To cover also this extension, we provide a different proof in Appendix \ref{sec:compactness} that is based on a characterization of compact sets.

Before diving into the proof, we introduce the \emph{modulus of continuity} of a measure $\mu \in \mathcal P_p(\mathcal A \times \mathcal B)$ by
\begin{equation}
    \label{eq:def.moc}
    \omega_\mu(\delta) :=
    \sup \left\{ \mathbb E[ d_{\mathcal B}(Y,\tilde Y)^p ]^\frac1p : (X,Y), (\tilde X, \tilde Y) \sim \mu, \, \mathbb E[d_{\mathcal A}(X,\tilde X)^p]^\frac1p \le \delta \right\} \quad \text{for }\delta \ge 0.
\end{equation}
Formally, this definition depends on $p$ but we suppress this dependency.

It is clear from the definition  that $\omega_\mu$ is increasing and it is also straightforward to establish  that $\omega_\mu(0)=0 $ if and only if $\mu$ is concentrated on the graph of a measurable function $f \colon \mathcal A \to \mathcal B$. A strengthening of this, which is crucial for our purposes, was shown in \cite[Lemma 2.7]{Ed19}:
There is a measurable function $f \colon \mathcal A \to \mathcal B$  such that $\mu$ is supported on the graph of $f$ if and only if 
\begin{align}\label{eq:mod_cont}
    \lim_{\delta\to 0} \omega_\mu(\delta)=0.
\end{align}

\begin{proof}[Proof of Theorem \ref{thm:Main}]
    Let $\mu,\nu \in \mathcal P_p(\mathbb R^N)$.
    We give the proof for $p \ge 1$ first and show that
    \begin{equation}
        \label{eq:direct.proof.0}
        \AW_p(\mu,\nu) \le \KR_p(\mu,\nu) \le \sum_{k = 1}^N f_\mu^k\left( \AW_p(\mu,\nu) \right),
    \end{equation}
    where $f_\mu^k\colon \mathbb R^+ \to \mathbb R^+$ is a continuous function that vanishes at 0.
    The first inequality follows directly from Lemma \ref{lem:KR.bicausal}.
    For $\delta \ge 0$ and $k \in \{2 ,\ldots,N\}$, we set
    \begin{align}\label{eq:SpecialCaseModulus}
        \omega^k_\mu(\delta) := \omega_{ (\id, K^\mu_k)_\# ( \proj^{1:k-1}_\# \mu)}(\delta).
    \end{align}
    To see the second inequality, we fix $k$ and consider the random variables $X = Q^\mu(U)$, $Y = Q^\nu(U)$ where $U \sim \lambda^N$.
    Pick $\tilde X$ such that $(\tilde X, Y) \sim \pi \in \Pi_{\rm bc}(\mu,\nu)$ and $\mathbb E[\| \tilde X_{1:k} - Y_{1:k} \|^p] = \AW_p(\mu_{1:k},\nu_{1:k})^p$.
    Then we have
    \begin{align} \label{eq:direct.proof.remark}
        \mathbb E\left[ |X_k - Y_k|^p \right]^\frac1p
        &=
        \mathbb E\left[ \mathcal W_p^p \left( K_k^\mu(X_{1:k-1}), K_k^\nu(Y_{1:k-1}) \right) \right]^\frac1p
        \\ \label{eq:direct.proof.remark0}
        &\le
        \mathbb E\left[ \mathcal W_p^p \left( K_k^\mu(X_{1:k-1}), K_k^\mu(\tilde X_{1:k-1}) \right) \right]^\frac1p
            +
        \mathbb E\left[ \mathcal W_p^p \left( K_k^\mu(\tilde X_{1:k-1}), K_k^\nu(Y_{1:k-1}) \right) \right]^\frac1p
        \\ \nonumber
        &\le
        \omega_\mu^k\left( \mathbb E\left[ | X_{1:k-1} - \tilde X_{1:k-1}|^p_p \right]^\frac1p \right) + \mathbb E\left[ |\tilde X_k - Y_k|^p \right]^\frac1p
        \\ \label{eq:direct.proof.1}
        &\le
        \omega_\mu^k\left( \mathbb E\left[ | X_{1:k-1} - Y_{1:k-1}|^p_p \right]^\frac1p + \AW_p(\mu,\nu) \right) + \AW_p(\mu,\nu),
    \end{align}
    where we used in \eqref{eq:direct.proof.remark} the fact that the monotone rearrangement is $\mathcal W_p$-optimal for $p \ge 1$, followed by the Minkowski's inequality in \eqref{eq:direct.proof.remark0}, then the definition of $\omega_\mu^k$ combined with optimality of $\pi$, and finally in \eqref{eq:direct.proof.1} the basic observation that $\AW_p(\mu_{1:k},\nu_{1:k}) \le \AW_p(\mu,\nu)$ as well as the Minkowski's inequality.
    We define inductively for $k = 2,\ldots,N$ and $\delta \ge 0$
    \begin{equation}
        \label{eq:direct.proof.2}
        f_\mu^1(\delta) := \delta \text{ and }
        f_\mu^k(\delta) := \omega_\mu^k\left(  \delta  + \sum_{\ell = 1}^{k-1} f_\mu^{\ell}(\delta)\right) + \delta.
    \end{equation}
    It follows from \cite[Lemma 2.7]{Ed19} that $(f_\mu^k)_{k = 1}^N$ is a family of continuous functions that vanish at $0$.
    Substituting with \eqref{eq:direct.proof.2} in \eqref{eq:direct.proof.1} leads us to
    \begin{equation}
        \label{eq:direct.proof.3}
        \mathbb E[|X_k-Y_k|^p]^\frac1p \le
        f_\mu^k(\AW_p(\mu,\nu)).
    \end{equation}
    Employing first  sub-additivity of $x\mapsto x^{1/p}$ and then \eqref{eq:direct.proof.3} we find    
    \begin{align*}
        \KR_p(\mu,\nu) 
        = 
        \mathbb E\left [|X - Y|_p^p \right]^\frac1p 
        \le 
        \sum_{k=1}^N \mathbb E\left[ |X_k - Y_k|^p \right]^\frac1p\le
        \sum_{k = 1}^N  f_\mu^k\left( \AW_p(\mu,\nu) \right),
    \end{align*}
    which yields \eqref{eq:direct.proof.0}.

    Let $(\mu^n)_{n \in \mathbb N}$ be a sequence in $\mathcal P_p(\mathbb R^N)$ and $\mu \in \mathcal P_p(\mathbb R^N)$.
    Hence, we derive from \eqref{eq:direct.proof.0} that $\mu^n \to \mu$ in $\AW_p$ if and only if $\mu^n \to \mu$ in $\KR_p$, since $f^k_\mu$ vanishes at $0$.
    This concludes the proof for $p \ge 1$.
    
    When $p = 0$, the equality in \eqref{eq:direct.proof.remark} may fail.
    Again, let $(\mu^n)_{n \in \mathbb N}$ be a sequence in $\mathcal P(\mathbb R^N)$ with $\AW_0$-limit $\mu \in \mathcal P(\mathbb R^N)$.
    Write $\pi^n \in \cplbc(\mu^n,\mu)$ for an $\AW_0$-optimal coupling.
    Since $\tanh \colon \R \to (0,1)$ is a continuous bijection, we have that $T(x) := (\tanh(x_1),\ldots,\tanh(x_N))$ is an adapted bijection with adapted inverse.
    By \cite[Lemma 3.16]{BePaSc22} we have that couplings induced by adapted bijections with adapted inverse are bicausal.
    This means that $(\id,T)_\# \mu \in \cplbc(\mu,T_\#\mu)$ and $(\id,T)_\# \mu^n \in \cplbc(\mu,T_\#\mu)$, from which follows that $T_\# \pi^n \in \cplbc(T_\# \mu^n, T_\# \mu)$ (as $T_\# \pi^n$ can be written as a gluing of bicausal coupling).
    Moreover, we have
    \[
        \lim_{n \to \infty} \int |T(x) - T(y)|_1 \, d\pi^n(x,y) \le \lim_{n \to \infty} 2 N \int |x - y| \wedge 1 \, d\pi^n(x,y) = 0,
    \]
    since $|\tanh(\hat x) - \tanh(\hat y)| \le |\hat x-\hat y|\wedge 2$ for $\hat x,\hat y \in \R$.
    This shows that $T_\# \mu^n \to T_\# \mu$ in $\AW_1$, thus, also in $\KR_1$.
    In terms of quantile processes, we have $T\circ Q^{\mu^n} \to T \circ Q^\mu$ in $\lambda^N$-probability, whence, $Q^{\mu^n} \to Q^\mu$ in $\lambda^N$-probability.
    We have shown that $\mu^n \to \mu$ in $\KR_0$.
    Since $\AW_0 \le \KR_0$, the reverse direction is trivial, which completes the proof.
\end{proof}


Let $p\geq 1$ and write $\P_{p,L}(\R^N)$ for the set of $\mu\in \P_{p,L}(\R^N)$ which are
\emph{triangular $L$-Lipschitz} in the sense that 
\[ K^\mu_k \colon \mathcal \R^{k-1} \to \mathcal{P}_p(\R) \]
is $L$-Lipschitz w.r.t.\ $p$-Wasserstein distance on $\mathcal{P}_p(\R)$. 
\begin{theorem}\label{thm:EquivalentMetrics}
    Let $p\geq 1$ and $L>0$. On $ \P_{p,L}(\R^N)$ the adapted Wasserstein distance and Knothe--Rosenblatt distance are equivalent. That is, there exists a constant $C>0$  such that for all $\mu, \nu \in \P_{p,L}(\R^N)$ 
    \[ \AW_p(\mu, \nu)\leq \KR_p(\mu, \nu) \leq C \cdot \AW_p(\mu, \nu). \]
\end{theorem}
\begin{proof}
    From the definition of the modulus of continuity it is clear that (using the notation $ \omega_\mu^k$ from \eqref{eq:SpecialCaseModulus} ) \[ \omega_\mu^k(\delta) \leq L \delta. \]
    Then claim then follows in view of \eqref{eq:direct.proof.2} and \eqref{eq:direct.proof.1}.
\end{proof}

The statement in \cite[Theorem 2.2]{BoKoMe05} asserts that when a sequence of absolutely continuous measures $\mu_n\in \mathcal P(\R^N), n \in \N$, converges to  $\mu$ in total variation, then the corresponding quantile processes $Q^{\mu_n} \to Q^\mu$ in probability.
As a consequence of Theorem \ref{thm:Main} we obtain the following sharp result:

\begin{corollary}\label{cor:Bogachev}
    Let $\mu_n, \mu \in \mathcal P(\mathbb R^N)$, $n \in \N$.
    Then $\mu_n \to \mu$ in $\AW_p$ to $\mu$ if and only if $Q^{\mu_n} \to Q^\mu$ in $L_p(\lambda^N;\mathbb R^N)$.
    In particular, this is the case when $\mu_n \to \mu$ in total variation and $\int |x|_p^p \, d\mu_n(x) \to \int |x|_p^p \, d\mu(x)$.
\end{corollary}

\begin{proof}
    The first assertion follows directly from Theorem \ref{thm:Main}.
    To see the second assertion, note that by \cite[Lemma 3.5]{EcPa22} 
    \[
        TV(\mu,\mu_n) \le AV(\mu,\mu_n) = \inf_{\pi \in \cplbc(\mu,\mu_n) } \pi(\{(x,y) \in \R^N \times \R^N : x \neq y \}) \le (2^{N-1} - 1) TV(\mu,\mu_n),
    \]
    i.e., convergence in total variation $TV$ already entails convergence w.r.t.\ the so-called adapted variation $AV$. 
    In turn, the adapted variation distance dominates the $\AW_0$-metric, which yields convergence of $\mu_n \to \mu$ w.r.t.\ $\AW_0$.
    Thus, we have by the first part that $Q^{\mu_n} \to Q^\mu$ in $\lambda^N$-probability.
    Hence, by convergence of their $p$-th moments, we find that $Q^{\mu_n} \to Q^\mu$ in $L_p(\lambda^N;\mathbb R^N)$.
\end{proof}

\section{Consequences of the isometry with $L_p(\lambda^N; \R^N)$}\label{sec:IsometryConsequences}
In this section we collect consequences of the isometry \[
        \iota \colon \mathcal P_p(\mathbb R^N) \to L_p(\lambda^N;\mathbb R^N) \colon \mu \mapsto Q^\mu
    \] (which we already mentioned in \eqref{eq:embedding}). It represents a very useful tool for the investigation of the Knothe--Rosenblatt distance as it allows to transfer properties from (convex subsets of) $L_p(\lambda^N;\mathbb R^N)$ to $\mathcal P_p(\mathbb R^N)$. As a first example, through the isometry $\iota$ it is evident that $\KR_p$ satisfies the triangular inequality, a fact that is otherwise rather tedious to prove.

\subsection{Triangular monotonicity}

Towards a better understanding of the metric completion of $\mathcal P_p(\mathbb R^N)$ equipped with $\KR_p$, we study the notion of triangular monotonicity of adapted functions $Q \colon (0,1)^N \to \mathbb R^N$.
Recall that $Q$ is called \KRmonotone\ if there exists a $\lambda^N$-full set $\Gamma \subseteq (0,1)^N$ such that for all $u,v \in \Gamma$ and $k \in \mathbb N$, $k \le N$, we have
\begin{equation}
    \label{eq:KRmonotone}
    u_{1:k-1} = v_{1:k-1} \text{ and }u_k \le v_k \implies Q_k(u_{1:k}) \le Q_k(v_{1:k}).
\end{equation}
If \eqref{eq:KRmonotone} remains additionally true when both $\le$ are replaced by $<$, then $Q$ is called strictly \KRmonotone.
Further, recall from the introduction that a measure $\mu \in \mathcal P(\mathbb R^N)$ is called \KRregular\ if for $\mu$-a.e.\ $x$ and $k \in \mathbb N$, $k \le N$,
\[
    K^\mu_{k}(x_1,\ldots,x_{k - 1}) \text{ is regular (i.e., is atom free)}.
\]
In the following we establish that the set $L_p^{\rm SIT}(\lambda^N;\R^N)$  of strictly \KRmonotone\ functions is closely related to the set of \KRregular\ probability measures.

\begin{proposition} \label{prop:Lp.monotone}
    The following statements hold:
    \begin{enumerate}
        \item The set of increasing triangular transformations $L_p^{\rm IT}(\lambda^N;\R^N)$ is a closed subset of $L_p(\lambda^N;\mathbb R^N)$;
        \item The set of strictly increasing triangular transformations $L_p^{\rm SIT}(\lambda^N;\R^N)$ is dense in $L_p^{\rm IT}(\lambda^N;\R^N)$.
    \end{enumerate}
\end{proposition}

\begin{proof}
    Clearly, the property of being \KRmonotone\ is closed under $L_p$-convergence, which shows the first claim.
    On the other hand, if $\epsilon > 0$, $Q$ is \KRmonotone, so is
    \[
        Q^\epsilon := Q + \epsilon \, \id.  
    \]
    Even more, $Q^\epsilon$ inherits from the identity map the property of being strictly \KRmonotone.
    Since $Q^\epsilon \to Q$ for $\epsilon \to 0$ in $L_p$, we deduce the second assertion.
\end{proof}

\begin{lemma} \label{lem:strictlyKRmonotone}
    Let $\mu \in \mathcal P(\mathbb R^N)$ be \KRregular.
    Then $Q^\mu \in L_p^{\rm SIT}(\lambda^N;\R^N)$ and $\{Q^\mu\} = \{ Q \in L^{\rm IT}(\lambda^N;\R^N) : Q_\# \lambda^N = \mu \}$.
\end{lemma}

\begin{proof}
    Assume that $Q$ is a \KRmonotone\ map with $Q_\# \lambda^N = \mu$.
    Let $U = (U_1,\ldots,U_N) \sim \lambda^N$ and $X = Q(U)$.
    Clearly, for $k = 1$ we have that $Q_1 = Q^\mu_1$.
    So, let us assume that $Q_{1:k-1} = Q^\mu_{1:k-1}$ $\lambda^N$-almost surely.
    As $Q^\mu$ is strictly \KRmonotone, this means that $\sigma(U_{1:k-1}) = \sigma(Q^\mu_{1:k-1}(U)) = \sigma(Q_{1:k-1}(U)) = \sigma(X_{1:k-1})$, whence,
    \[  
        K^\mu_{k}(X_{1:k-1}) = Law( X_k | X_{1:k-1}) = Law( X_k | U_{1:k-1}). 
    \]
    Since $Q$ is \KRmonotone\ and $\mu$ \KRregular, we find that $Q_k(U_{1:k}) = Q^{K^\mu_k(X_{1:k-1})}(U_k) = Q^\mu(U_{1:k})$ almost surely.
    We conclude that $Q = Q^\mu$ $\lambda^N$-almost surely.
\end{proof}

\begin{theorem} \label{thm:props.embedding}
    The isometry
     $
        \iota \colon \mathcal P(\mathbb R^N) \to L(\lambda^N;\mathbb R^N) \colon \mu \mapsto Q^\mu.
     $
    has the following properties:
    \begin{enumerate}[label = (\arabic*)]
        \item \label{it:thm.props.embedding.2} $\iota( \{ \mu \in \mathcal P(\mathbb R^N) : \mu \text{ is \KRregular} \} ) = L_p^{\rm SIT}(\lambda^N;\R^N)$;
        \item \label{it:thm.props.embedding.3} Via the isometry $\iota$, the metric completion of $(\mathcal P_p(\mathbb R^N), \KR_p)$ is $L_p^{\rm IT}(\lambda^N;\R^N)$;
        \item \label{it:thm.props.embedding.4} Restricted to $L_p^{\rm SIT}(\lambda^N;\R^N)$, the map $Q \mapsto Q_\# \lambda^N$ is the inverse of $\iota$.
    \end{enumerate}
\end{theorem}

\begin{proof}

    \ref{it:thm.props.embedding.2}:
    On the one hand, if $\mu$ is regular, then we have by Lemma \ref{lem:strictlyKRmonotone} that $Q^\mu \in L^{\rm SIT}_p(\lambda^N;\R^N)$.
    On the other hand, if $Q \in L^{\rm SIT}_p(\lambda^N;\R^N)$, then it is straightforward to verify that $Q$ admits an adapted, almost sure right-inverse, i.e., an adapted map $R \colon \mathbb R^N \to (0,1)^N$ with $R \circ Q = \id$ $\lambda^N$-almost surely.
    Thus, the induced Monge coupling satisfies $(\id,Q)_\# \lambda^N = (R,\id)_\# \mu$ (where $\mu := Q_\# \lambda^N$).
    In \cite[Lemma 3.16]{BePaSc21c} it was shown that couplings induced by adapted bijections with adapted inverse, are bicausal, whence, $(\id,Q)_\# \lambda^N \in \cplbc(\lambda^N,\mu)$.
    But, in \cite[Proposition 5.8]{BaBeLiZa17} it was established that $Q^\mu$ is the unique element of $L^{\rm IT}_p(\lambda^N;\R^N)$ with $(\id,Q^\mu)_\# \lambda^N \in \cplbc(\lambda^N;\mu)$.
    Hence, $Q = Q^\mu$, which also proves \ref{it:thm.props.embedding.4}.
    
    \ref{it:thm.props.embedding.3}: Since for every $\mu \in \mathcal P(\mathbb R^N)$, the quantile process is \KRmonotone, the assertion is a consequence of Proposition \ref{prop:Lp.monotone}.
\end{proof}

It is easy to see that the set of increasing triangular transformations is closed, but, at least to the authors, it is not immediate whether or not $L^{quantile}_p(\lambda^N;\R^N) = \{ Q^\mu : \mu \in \mathcal P_p(\mathbb R^N) \}$ is $G_\delta$.
This question is positively answered in the subsequent corollary.
It is somewhat remarkable that by a small detour in the reasoning this question becomes much easier to answer.

\begin{corollary}
    The set $L^{\rm quantile}_p(\lambda^N;\R^N)$ is a $G_\delta$-subset of $L_p(\lambda^N;\mathbb R^N)$.
\end{corollary}

\begin{proof}
    The set $L^{\rm quantile}_p(\lambda^N;\R^N)$ is homeomorphic to $\mathcal P_p(\mathbb R^N)$ when the latter is endowed with the $p$-th Knothe--Rosenblatt topology. 
    By Theorem \ref{thm:AW.Polish}, the space $\mathcal P_p(\mathbb R^N)$ equipped with the $p$-adapted Wasserstein topology is Polish.
    Since the $p$-Knothe--Rosenblatt topology and the $p$-adapted Wasserstein topology coincide (cf.\ Theorem \ref{thm:Main}), we have  that also $\{ Q^\mu : \mu \in \mathcal P_p(\mathbb R^N) \}$ is a Polish space.
    Hence, we conclude by recalling that a subset of a Polish space endowed with the trace topology is itself Polish if and only if it is $G_\delta$, see \cite[Theorem 3.11]{Ke95}.
\end{proof}

\subsection{Geodesics}

This section is concerned with providing the missing bits to complete Corollary \ref{cor:GeoComplete}.
Based on the isometry \eqref{eq:embedding} which was shown in \Cref{thm:props.embedding}, it is sufficient to establish the following lemma: 
\begin{lemma}
    The sets $L_p^{{\rm SIT}}(\lambda^N;\R^N), L_p^{{\rm quantile}}(\lambda^N;\R^N), L_p^{{\rm IT}}(\lambda^N;\R^N) $  are convex.
\end{lemma}

\begin{proof}
    The only non-trivial claim is convexity of $L_p^{{\rm quantile}}(\lambda^N;\R^N)$ which we proceed to show.
    Assume that $T,S \colon (0,1)^N \to \R^N$ satisfy \eqref{CanonicalIncreasingI} and \eqref{CanonicalIncreasingII} and let $\alpha \in (0,1)$.
    We claim that $(1 - \alpha) T + \alpha S$ also satisfy \eqref{CanonicalIncreasingI} and \eqref{CanonicalIncreasingII}.
    Clearly, $(1 - \alpha) T + \alpha S$ satisfies \eqref{CanonicalIncreasingI}, so that it remains to prove \eqref{CanonicalIncreasingII}.
    
    To this end, fix $k< N$ and $u_1,\ldots, u_N, u_k'\in (0,1)$ with $u_k < u_k'$ and assume that
    \begin{equation}
        \label{eq:convex0}
        \left( (1 - \alpha) T + \alpha S \right)_k (u_1,\ldots,u_k) 
        = 
        \left( (1 - \alpha) T + \alpha S \right)_k (u_1,\ldots,u_k').
    \end{equation}
    If either
    \begin{equation}
        \label{eq:convex1}
        T(u_1, \ldots, u_k) < T(u_1, \ldots, u_k') \quad\text{or}\quad S(u_1, \ldots, u_k) < S(u_1, \ldots, u_k'),
    \end{equation}
    then also 
    \[ ( (1- \alpha) T+ \alpha S)(u_1, \ldots, u_k) < ( (1-\alpha) T+ \alpha S)(u_1, \ldots, u_k'), \]
    which contradicts \eqref{eq:convex0}.
    Therefore, \eqref{eq:convex0} implies that
    \begin{equation}
        \label{eq:convex2}
        T(u_1, \ldots, u_k) = T(u_1, \ldots, u_k')\quad\text{and}\quad S(u_1, \ldots, u_k) = S(u_1, \ldots, u_k').
    \end{equation}
    Under \eqref{eq:convex2} consistency of $S$ and $T$ yields
    \[ ((1-\alpha) T+ \alpha S)(u_1, \ldots, u_k, \ldots, u_m) = ((1-\alpha) T+ \alpha S)(u_1, \ldots, u_k', \ldots, u_m). \]
    We have shown that \eqref{CanonicalIncreasingII} holds for $(1-\alpha) T+ \alpha S$, which concludes the proof.
\end{proof}

\subsection{On the preservation of probabilistic properties along geodesics} \label{ssec:preservative}
Assume that $p\in (1, \infty)$. Then $\KR_p$ geodesics are unique as consequence of the isometry \eqref{eq:embedding}. Every such geodesic $(\mu_t)_{t\in [0,1]}$ is given by the McCann-interpolation induced by a Knothe--Rosenblatt coupling $\pi\in \cplbc(\mu_0, \mu_1)$, i.e.\
\begin{align}\label{eq:McCannStyle} \mu_t=((1-t) \proj^1 + t \proj^2)_{\#} \pi.\end{align}
The sets of martingales and predictable processes are closed under interpolations of the form \eqref{eq:embedding} for bicausal $\pi$, see e.g.\ \cite[Section 5]{BaBePa21}.
As in the case of adapted Wasserstein geodesics, Knothe--Rosenblatt geodesics do not preserve the Markov property.

\begin{example}
    Let $\mu_0 = \frac12(\delta_{(1,1,0)} + \delta_{(0,0,1)})$ and $\mu_1= \frac12 (\delta_{(1,0,0)} + \delta_{(0,1,1)})$.
    Then we have
    \begin{align*}
        Q^\mu(u) &= \mathbbm 1_{(0,\frac12)}(u_1) (0,0,1) + \mathbbm 1_{[\frac12,1)}(u_1)(1,1,0), \\
        Q^\nu(u) &= \mathbbm 1_{(0,\frac12)}(u_1) (0,1,1) + \mathbbm 1_{[\frac12,1)}(u_1)(1,0,0).
    \end{align*}
    The quantile process $Q$ corresponding to the midpoint $\mu_{\frac12}$ of $\mu$ and $\nu$ is given by
    \[
        Q(u) = \frac12 (Q^\mu(u) + Q^\nu(u)) = \mathbbm 1_{(0,\frac12)}(u_1) \left(0, \frac12,1\right) + \mathbbm 1_{[\frac12,1)}(u_1)\left(1,\frac12,0\right).
    \]
    The law $\mu_{\frac12}$ is given by $\frac12(\delta_{(0,\frac12,1)} + \delta_{(1,\frac12,0)})$, which is clearly the law of a non-Markovian process.
\end{example}

\subsection{Comparison of the $\KR$ and the  $\AW$-completion}\label{sec:DifferentComp}

Although $\AW_p$ and $\KR_p$ are topologically equivalent, rather unsurprisingly their metric completions do not coincide. 
Recall that up to isometry, the completion of a metric space consists of all Cauchy sequences (where two Cauchy-sequences are identified if mixing them yields again a Cauchy-sequence). 

In this section we will show that the completions w.r.t.\ $\AW_p$ and $\KR_p$ are mutually non comparable by establishing the following two Propositions: 

\begin{proposition}\label{prop:compare1} Let $p\in \{0\} \cup [1,\infty)$ and $N \ge 2$.
There is an $\AW_p$-Cauchy sequence $\mu_n\in \P_p(\R^N), n\geq 1$ which contains no $\KR_p$-Cauchy sequence, i.e.\ $\inf\{\KR_p(\mu_k, \mu_n): k\neq n\} >0$. 
\end{proposition}

Intuitively speaking, Proposition \ref{prop:compare1} asserts that the $\AW_p$-completion has a point which is not in the $\KR_p$-completion. In particular,
\[ \overline{\P_p(\R^N)}^{\AW_p} \not \subseteq \overline{\P_p(\R^N)}^{\KR_p}. \]
Rigorously this means that there is no continuous mapping from $\overline{\P_p(\R^N)}^{\KR_p}$ onto $\overline{\P_p(\R^N)}^{\AW_p}$ which is the identity on $\P_p(\R^N)$.

\begin{proposition}\label{prop:compare2} Let $p\in \{0\} \cup [1,\infty)$ and $N \ge 2$.
There is an $\AW_p$-Cauchy sequence $\mu_n\in \P_p(\R^N), n\geq 1$ which is not a $\KR_p$-Cauchy sequence but contains but contains two $\KR_p$-Cauchy sequences.      
\end{proposition}

Again on an intuitive level, Proposition \ref{prop:compare2} asserts that  there are points $p,p'$ in the $\KR_p$-completion which are identified in the $\AW_p$-completion. In particular $$\overline{\P_p(\R^N)}^{\KR_p} \not \subseteq \overline{\P_p(\R^N)}^{\AW_p} .$$ Rigorously this means that there is no continuous mapping from $\overline{\P_p(\R^N)}^{\AW_p}$ onto $\overline{\P_p(\R^N)}^{\KR_p}$ which is the identity on $\P_p(\R^N)$.

\begin{proof}[Proof of Proposition \ref{prop:compare1}]
    Without loss of generality we show the assertion for $N = 2$ and remark that the general statement (for $N \ge 2$) simply follows by adequately embedding.
    Fix $n\in \N$. For $u\in (0,1)$ write $d_n(u)$ for the $n$-th digit in the binary digit expansion of $u$.
    Define the quantile process $Q^n$ by $Q^n_1(u_1)=u_1, Q^n_2 (u_1, u_2)= d_n(u_1)$ and let $\mu_n:= Q^n_{\#}(\lambda^2).$
   
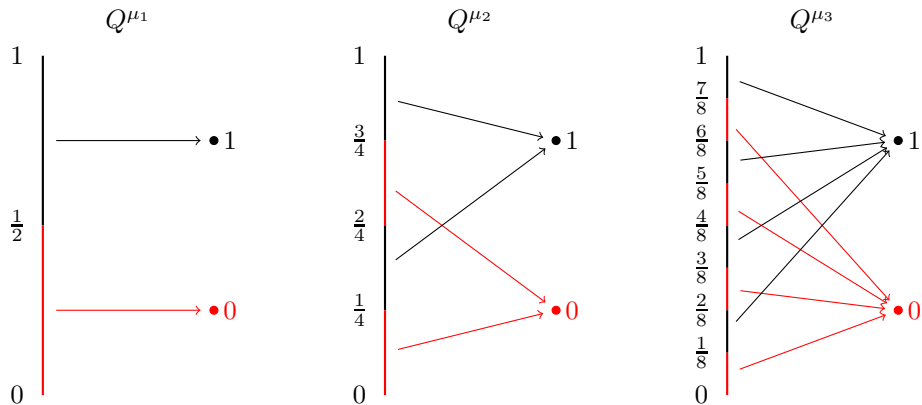
\begin{figure}[htb!]
    \centering
    
    \begin{tikzpicture}[scale=2.25]
    \node at (0.5,1.2) {$Q^{\mu_1}$};
        \filldraw[black] (1,0.5) circle (0.666pt) node[anchor=west]{1};
        \filldraw[red] (1,-0.5) circle (0.666pt) node[anchor=west]{0};

        \node at (-0.15,1) {$1$};
        \node at (-0.15,0) {$\frac{1}{2}$};
        \node at (-0.15,-1) {$0$};
        
        \draw[black,thick](0,0)--(0,1);
        \draw[red,thick](0,0)--(0,-1);
        
        \draw[black,->,shorten >=5pt,shorten <=5pt] (0,0.5)--(1,0.5);
        \draw[red,->,shorten >=5pt,shorten <=5pt] (0,-0.5)--(1,-0.5);

     \node at (0.5+2,1.2) {$Q^{\mu_2}$};
        \filldraw[black] (1+2,0.5) circle (0.666pt) node[anchor=west]{1};
        \filldraw[red] (1+2,-0.5) circle (0.666pt) node[anchor=west]{0};

        \node at (-0.15+2,1) {$1$};
        \node at (-0.15+2,0.5) {$\frac{3}{4}$};
        \node at (-0.15+2,0) {$\frac{2}{4}$};
        \node at (-0.15+2,-0.5) {$\frac{1}{4}$};
        \node at (-0.15+2,-1) {$0$};
        
        \draw[black,thick](0+2,1)--(0+2,0.5);
        \draw[red,thick](0+2,0.5)--(0+2,0);
        \draw[black,thick](0+2,0)--(0+2,-0.5);
        \draw[red,thick](0+2,-0.5)--(0+2,-1);
        
        \draw[black,->,shorten >=5pt,shorten <=5pt] (0+2,0.75)--(1+2,0.5);
        \draw[red,->,shorten >=5pt,shorten <=5pt] (0+2,0.25)--(1+2,-0.5);
        \draw[black,->,shorten >=5pt,shorten <=5pt] (0+2,-0.25)--(1+2,0.5);
        \draw[red,->,shorten >=5pt,shorten <=5pt] (0+2,-0.75)--(1+2,-0.5);x`

    \node at (0.5+4,1.2) {$Q^{\mu_3}$};
        \filldraw[black] (1+4,0.5) circle (0.666pt) node[anchor=west]{1};
        \filldraw[red] (1+4,-0.5) circle (0.666pt) node[anchor=west]{0};

        \node at (-0.15+4,1) {$1$};
        \node at (-0.15+4,0.75) {$\frac{7}{8}$};
        \node at (-0.15+4,0.5) {$\frac{6}{8}$};
        \node at (-0.15+4,0.25) {$\frac{5}{8}$};
        \node at (-0.15+4,0) {$\frac{4}{8}$};
        \node at (-0.15+4,-0.25) {$\frac{3}{8}$};
        \node at (-0.15+4,-0.5) {$\frac{2}{8}$};
        \node at (-0.15+4,-0.75) {$\frac{1}{8}$};
        \node at (-0.15+4,-1) {$0$};
        
        \draw[black,thick](0+4,1)--(0+4,0.75);
        \draw[red,thick](0+4,0.75)--(0+4,0.5);
        \draw[black,thick](0+4,0.5)--(0+4,0.25);
        \draw[red,thick](0+4,0.25)--(0+4,0);
        \draw[black,thick](0+4,0)--(0+4,-0.25);
        \draw[red,thick](0+4,-0.25)--(0+4,-0.5);
        \draw[black,thick](0+4,-0.5)--(0+4,-0.75);
        \draw[red,thick](0+4,-0.75)--(0+4,-1);
        
        \draw[black,->,shorten >=5pt,shorten <=5pt] (0+4,0.875)--(1+4,0.5);
        \draw[red,->,shorten >=5pt,shorten <=5pt] (0+4,0.625)--(1+4,-0.5);

        \draw[black,->,shorten >=5pt,shorten <=5pt] (0+4,0.375)--(1+4,0.5);
        \draw[red,->,shorten >=5pt,shorten <=5pt] (0+4,0.125)--(1+4,-0.5);

        \draw[black,->,shorten >=5pt,shorten <=5pt] (0+4,-0.125)--(1+4,0.5);        
        \draw[red,->,shorten >=5pt,shorten <=5pt] (0+4,-0.375)--(1+4,-0.5);

        \draw[black,->,shorten >=5pt,shorten <=5pt] (0+4,-0.625)--(1+4,0.5);
        \draw[red,->,shorten >=5pt,shorten <=5pt] (0+4,-0.875)--(1+4,-0.5);
    
    \end{tikzpicture}
    
    \caption{Illustration  of the quantile processes of $\mu_1,\mu_2,\mu_3$ in Proposition \ref{prop:compare1}.}
    \label{fig:precompactness}
\end{figure}

Then we have for $p\geq 1$
\[
    \KR_p(\mu^n,\mu^m) = \mathbb E[|X^n - X^m|_p^p]^\frac1p = \mathbb P[ X^n \neq X^m]^\frac1p 
    = \begin{cases}
        0 & n = m, \\
        \left( \frac12 \right)^\frac1p & n \neq m.
    \end{cases}
\]
and similarly for $p=0$. However, it is easy to see that $(\mu_n)_{n \in \N}$ constitutes an $\AW_p$-Cauchy sequence.
\end{proof}

\begin{proof}[Proof of Proposition \ref{prop:compare2}]
    Again, we can assume without loss of generality that $N = 2$.
    We define the quantile processes $Q^n(u) = (Q^1(u_1),Q^2(u_1,u_2))$ via
\[
        Q^n_1(u_1) := \frac{\text{sgn}{(u_1-1/2)}}{n}\quad\text{and}\quad
        Q^n_2(u_1,u_2) := (-1)^n \text{sgn}{(u_1-1/2)},
    \]    
    where $u = (u_1,u_2) \in (0,1)^2$.
    Since $Q^n$ satisfies \eqref{CanonicalIncreasingI} and \eqref{CanonicalIncreasingII} we have by Lemma \ref{lem:Q.defs} that $Q^n \in L_p^{\rm quantile}(\lambda^2;\R^2)$.
    We write $\mu_n = Q^n_\# (\lambda^2)$.
    It is then straightforward to see that and
    \begin{enumerate}
        \item $(\mu_n)_{n \in \N}$ constitute an $\AW_p$-Cauchy sequence;
        \item $(\mu_{2n})_{n \in \N}$ and $(\mu_{2n+1})_{n \in \N}$ are both $\KR_p$-Cauchy sequences;
        \item $(\mu_n)_{n \in \N}$ is not a $\KR_p$-Cauchy sequence.
    \end{enumerate}
    Hence, we have found a sequence with the desired properties.
\end{proof}

\appendix

\section{Knothe--Rosenblatt in multiple dimensions}

\subsection{On the weak topology and convergence in probability}
We denote by $J \colon \mathcal P(\mathcal A \times \mathcal B) \to \mathcal P(\mathcal A \times \mathcal P(\mathcal B))$ the (measurable) map satisfying 
\[
    J(\mu) = \left( (a,b) \mapsto (a,K^\mu_2(a)) \right)_\#\mu.
\]
The set of all probabilities $\mu$ concentrated on the graph of a measurable function is denoted by $\mathcal F_p(\mathcal A \leadsto \mathcal B)$.

\begin{lemma} \label{lem:graph.Polish}
    The set $\mathcal F_p(\mathcal A \leadsto \mathcal B)$ is a $G_\delta$-subset of $\mathcal P_p(\mathcal A \times \mathcal B)$ for $p\in \{0\} \cup [1,\infty)$.
\end{lemma}

\begin{proof}
    As in Section \ref{sec:ProofMainResult} we appeal to \cite[Lemma 2.7]{Ed19} which asserts that a measure $\mu \in \mathcal P_p(\mathcal A \times \mathcal B)$ is in $\mathcal F_p(\mathcal A \leadsto \mathcal B)$ if and only if its modulus of continuity  satisfies $\lim_{\delta \searrow 0} \omega_\mu(\delta) = 0$.
    Hence, we have
    \begin{align*}
        \mathcal F_p(\mathcal A \leadsto \mathcal B) &=
        \bigcap_{\epsilon > 0} \bigcup_{\delta > 0} \left\{ \mu \in \mathcal P_p(\mathcal A \times \mathcal B) : \omega_\mu(\delta) < \epsilon \right\}.
    \end{align*}
     \cite[Lemma 2.9]{Ed19} establishes that for fixed $\delta > 0$, the map $\mu \mapsto \omega_\mu(\delta)$ is continuous on $\mathcal P_p(\mathcal A \times \mathcal B)$. We thus conclude that $\mathcal F_p(\mathcal A \leadsto \mathcal B)$ is the countable intersection of open sets, that is, $G_\delta$.
\end{proof}

Recall that $\mathcal X = \mathcal X_{1:N} = \prod_{k = 1}^N \mathcal X_k$ is the product of $N$ Polish metric spaces $(\mathcal X_k, d_{\mathcal X_k})$ and that $\AW_p$ denotes the $p$-adapted Wasserstein distance as defined in \eqref{eq:def.AW}.

\begin{theorem} \label{thm:AW.Polish}
    The space $(\mathcal P_p(\mathcal X), \AW_p)$ is Polish for $p\in \{0\} \cup [1,\infty)$.
\end{theorem}

\begin{proof}
    We show the claim by induction over $N$.
    When $N = 1$ we have that $\AW_p=\W_p$, which yields the claim.
    Next, let $N > 1$ and assume that $\mathcal Y_N := (\mathcal P_p(\mathcal X_{2:N}), \AW_p)$ is Polish.
    To clarify, in this particular instance $\AW_p$ denotes the $p$-adapted Wasserstein distance for measures on the path space $\tilde{\mathcal X} := \tilde{\mathcal X}_{1:N-1}$ where $\tilde{\mathcal X}_k = \mathcal X_{k + 1}$, $k < N$.
    Consider the map $\tilde J$ given by
    \begin{align*}
        \tilde J \colon \mathcal P_p(\mathcal X)
        &\to \mathcal P_p\left(\mathcal X_1 \times \mathcal Y_N \right), \\
        \mu &\mapsto \left( x \mapsto (x_1, \mu^{x_1}) \right)_\# \mu,
    \end{align*}
    where $(\mu^{x_1})_{x_1 \in \mathcal X_1}$ denotes the regular disintegration of $\mu$ w.r.t.\ the $\mathcal X_1$-coordinate.
    Clearly, $\tilde J$ is injective with range $\mathcal F_p(\mathcal X_1 \leadsto \mathcal Y_N)$.
    Moreover, it is not hard to see that $\tilde J$ is a topological embedding.
    Indeed,  \cite[Theorem 3.10]{BaBePa21} asserts that for $\mu,\nu \in \mathcal P_p(\mathcal X)$ we have the isometry
    \[
        \AW_p(\mu,\nu) = \widetilde{\W}_p(J(\mu),J(\nu)),
    \]
    where, for $\xi,\zeta \in \mathcal P_p(\mathcal X_1 \times \mathcal Y_N)$, we have the non-complete metric
    \[
        \widetilde{\W}_p^p(\xi,\zeta) := \inf_{\pi \in \cpl(\xi,\zeta)} \int d_{\mathcal X_1}^p(x_1,y_1) + \AW_p^p(\eta, \theta) \, \pi(d(x_1,\eta), d(y_1, \theta)).
    \]
    It is well-known that the topology induced by $\widetilde{\W}_p$ is equivalent to weak convergence and convergence of $p$-moments.
    Therefore, $J$ is an isometry from $(\mathcal P_p(\mathcal X),\AW_p)$ to $(\mathcal P_p(\mathcal X_1 \times \mathcal Y_N),\widetilde{\W}_p)$ and homeomorphic to $\mathcal F_p(\mathcal X_1 \leadsto \mathcal Y_N)$.
    The latter set is a $G_\delta$-subset of $\mathcal P_p(\mathcal X_1 \times \mathcal Y_N)$ by Lemma \ref{lem:graph.Polish}.
    Hence, $(\mathcal P_p(\mathcal X),\AW_p)$ is a Polish space.
\end{proof}

\begin{lemma} \label{lem:MeyerZheng}
    Let $\gamma \in \mathcal P(\mathcal A)$ and $p\in \{0\} \cup [1,\infty)$.
    The mapping 
    \[
        L_p(\gamma; \mathcal B) \to  \{ \mu \in F_p(\mathcal A \leadsto \mathcal B) : \proj^1_\# \mu = \gamma \} : f \mapsto (\id,f)_\# \gamma,
    \]
    is a homeomorphism.
\end{lemma}

\begin{proof}
    If $f^k \to f$ in $L_p(\gamma;\mathcal B)$ then $(\id,f^k)_\# \gamma \to (\id,f)_\# \gamma$ in $\mathcal P_p(\mathcal A \times \mathcal B)$, since $g \mapsto (\id,g)_\# \gamma$ is 1-Lipschitz from $L_p(\gamma;\mathcal B)$ to $\mathcal P_p(\mathcal A \times \mathcal B)$.
    The reverse direction was shown in \cite[Lemma 3.14]{JoPa23}.
\end{proof}

On the real line it is well-known that the map $\rho \mapsto Q^\rho$ is continuous with domain $\mathcal P(\R)$ and co-domain $L^0(0,1)$.
To cover also the multi-dimensional Knothe--Rosenblatt rearrangement given  in \Cref{ssec:generalization}, cf.\ \Cref{def:quantile_process_multi}, we a give proof of the following generalization of this fact.

\begin{corollary} \label{cor:Brenier.convergence}
    Let $p > 1$ and $S \colon \mathcal P_p(\mathbb R^d) \to L^p(\lambda^N;\mathbb R^d)$ be a selector of $\mathcal W_p$-optimal transport maps between $\lambda^d$ and $\rho$.
    Then $S$ is continuous.
\end{corollary}

\begin{proof}
    By \cite[Theorem 3.7]{GaMc96} we know that $(\id,S(\rho))_\# \lambda^d$ is the unique $\mathcal W_p$-optimal coupling from $\lambda^d$ to $\rho$.
    Hence, we deduce from stability of optimal transport, see e.g.\ \cite[Theorem 5.20]{Vi09}, that whenever $\rho^k \to \rho$ in $\mathcal P_p(\mathbb R^d)$ we also have $(\id,S(\rho^k))_\# \lambda^d \to (\id,S(\rho))_\# \lambda^d$ in $\mathcal P_p(\mathbb R^d \times \mathbb R^d)$.
    Then the assertion follows from Lemma \ref{lem:MeyerZheng}.
\end{proof}

\subsection{On compact sets in the Knothe--Rosenblatt topology}
\label{sec:compactness}

In this section we provide an alternative proof of \Cref{thm:Main} which is also applicable in the multi-dimensional setting introduced in \Cref{ssec:generalization}.
We begin by providing a characterization of relatively compact sets:

\begin{proposition} \label{prop:AW=KR} 
    Let $\mathcal K \subseteq \mathcal P_p(\mathbb R^N)$, $p\in \{0\} \cup [1,\infty)$. 
    Then the following are equivalent:
    \begin{enumerate}[label = (\arabic*)]
        \item \label{it:prop.AWKR.1} $\mathcal K$ is $\KR_p$-relatively compact;
      
        \item \label{it:prop.AWKR.5} $\mathcal K$ is $\AW_p$-relatively compact.
    \end{enumerate}
\end{proposition}

\begin{proof}
    By \Cref{lem:KR.bicausal} we have $\KR_p \ge \AW_p \ge \W_p$, so that \ref{it:prop.AWKR.1}$\implies$\ref{it:prop.AWKR.5}.

    To see the reverse implication, let $(\mu_n)_{n \in \N}$ be an $\AW_p$-convergent sequence with limit $\mu \in \mathcal P_p(\R^N)$.
    Write $X^n = Q^{\mu_n}(U)$, $n \in \N$, and $X = Q^\mu(U)$ where $U$ is uniformly distributed on $(0,1)^N$.
    It suffices to show that $X^n \to X$ in probability. 
    We proceed to show this claim by  induction:
    For $k = 1$ we have that $X^n_1 = Q^{K^{\mu_n}_1}(U_1) \to Q^{K^\mu_1}(U_1) = X_1$ almost surely, since $\mu_n \to \mu$ in $\W_p$ and therefore $K^{\mu_n}_1 = \proj^1_\# \mu_n \to \proj^1_\# \mu = K^\mu_1$ in $\mathcal P_p(\R)$.
    Next, assume that $X^n_{1:k-1} \to X_{1:k-1}$ in probability.
    It was established in \cite[Theorem 11]{Pa22} that the $\AW_p$-topology can be characterized as the initial topology w.r.t.\ the family of maps
    \[
        \mathcal P_p(\R^N) \to \mathcal P_p(\R^{k-1} \times \mathcal P_p(\R)) \colon \mu \mapsto (x \mapsto (x_{1:k-1}, K_k^\mu(x_{1:k-1}))),
    \]
    for $k = 1, \ldots, N$.
    Hence, $\mu^n \to \mu$ in $\AW_p$ entails convergence of
    \[
        (X^n_{1:k-1}, Law(X^n_k | X^n_{1:k-1})) \to (X_{1:k-1}, Law(X_k | X_{1:k-1})) \quad \text{in law}.
    \]
    In particular, we get by independence of $U_k$ that
    \[
        (X^n_{1:k-1},U_k,Law(X^n_k| X^n_{1:k-1}) \to (X_{1:k-1}, U_k, Law(X^n_k | X_{1:k-1})) \quad \text{in law}.
    \]
    Recall that for a sequence $(\rho_n)_{n \in \N}$ in $\mathcal P(\R)$ with $\rho_n \to \rho$ in the $p$-weak topology, we have that $Q^{\rho_n}(U_k) \to Q^{\rho}(U_k)$ w.r.t.\ $\|.\|_p$.
    Therefore, we have that
    \[
        \pi_n := Law(X^n_{1:k-1},U_k,X^n_k ) \to Law(X_{1:k-1},U_k,X_k) =: \pi,
    \]
    in $\mathcal P_p(\R^{1:k-1} \times (0,1) \times \R)$, where we used that $X^n_k = Q^{Law(X^n_k | X_{1:k-1})}(U_k)$ and $X_k = Q^{Law(X_k | X_{1:k-1})}(U_k)$.
    At the same time, we have by the inductive assumption that $\gamma_n := Law(U_{1:k-1}, X_{1:k-1}^n)$, $n \in \N$, converges to $Law(U_{1:k-1}, X_{1:k-1}) =: \gamma$ in $\mathcal P_p((0,1)^{k-1}\times\R^{k-1})$.
    Observe that $\pi \in \mathcal F_p(\R^{k-1} \times (0,1) \leadsto \R)$.
    Thus, by continuity of gluing at $(\gamma,\pi)$, see \cite[Theorem 4.1]{Ed19}, we have that
    \[
        \gamma^n \otimes K^{\pi^n}_1 \to \gamma \otimes K^{\pi}_1 \text{ in }\mathcal P_p((0,1)^{k-1} \times \mathbb R^{k-1} \times (0,1) \times \mathbb R),
    \]
    where we view $(x_{1:k-1},u_k) \in  \mathbb R^{k-1} \times (0,1)$ as the first coordinate of both, $\pi^n$ and $\pi$.
    Put differently, we have shown that in distribution
    \[
        (U_{1:k},X^n_{1:k}) \to (U_{1:k},X_{1:k}).
    \]
    Therefore, we can invoke Lemma \ref{lem:MeyerZheng} to obtain that $X^n_{1:k} \to X_{1:k}$ in probability, which concludes the induction step.
\end{proof}

\begin{proof}[Second proof of Theorem \ref{thm:Main}]
    By Proposition \ref{prop:AW=KR} we have that a set $\mathcal K \subseteq \mathcal P_p(\mathbb R^N)$ is $\AW_p$-relatively compact if and only if it is $\KR_p$-relatively compact.
    Since the topologies induced by $\KR_p$ and $\AW_p$ are both sequential as well as Hausdorff, this means that the topologies coincide.
\end{proof}

\subsection{Guide to the case $d>1$.}

In this section, we elaborate on the claims of Section \ref{ssec:generalization}.
We omit giving detailed proofs as most are analogue to the case $d=1$ with the obvious modifications.
Nevertheless, to resolve possible doubts we included a sketch of the arguments below.

\begin{proof}[Sketch of proof of Theorem \ref{thm:KRandAW}]
    Recall Definition \ref{def:quantile_process_multi} given in Section \ref{ssec:generalization}.
    Further, recall that by Corollary \ref{cor:Brenier.convergence} the unique selector $S \colon \P_p(\R^d) \to L_p(\lambda^d;\R^d)$ where $S(\rho)$ is the $\cW_p$-optimal transport map from $\lambda^d$ to $\rho$, is continuous.
    Analogous to Lemma \ref{lem:Q.defs} we introduce
    \[
        Q^\mu_k(u_{1:k}) := S^{K^\mu_k(Q^\mu_{1:k-1}(u_{1:k-1})},
    \]
    and find that $Q^\mu$ satisfies \eqref{CanonicalIncreasingIIMulti2} and \eqref{CanonicalIncreasingIMultiP}.
    Reasoning as in the proof of Lemma \ref{lem:Q.defs}, it follows that $Q \in L_p((\lambda^d)^N;(\R^d)^N)$ satisfies 
    \eqref{CanonicalIncreasingIIMulti2}, \eqref{CanonicalIncreasingIMultiP}, and $Q_\# (\lambda^d)^N = \mu$ if and only if $Q^\mu = Q$.
    We conclude that for each $\mu$ there is a unique quantile process and $\mu \mapsto Q^\mu$ is an isometry from $(\P_p((\R^d)^N),\KR_p)$ to $(L_p((\lambda^d)^N;(\R^d)^N),\|.\|_p)$.

    For $\mu,\nu \in \P_p((\R^d)^N)$ the multi-dimensional Knothe--Rosenblatt coupling $\kr_{\mu,\nu}$ is given by $(Q^\mu,Q^\nu)_\# (\lambda^d)^N$.
    Therefore, the successive disintegration of $\kr_{\mu,\nu}$ is $\kr_{\mu,\nu}$-almost surely given by
    \[
        K^{\kr_{\mu,\nu}}_k((x_\ell,y_\ell)_{\ell = 1}^{k-1}) 
        = Law
        \left( S(K^\mu_k(x_{1:k-1})(U), S(K^\nu_k(y_{1:k-1})(U)) | ((x_\ell,y_\ell)_{\ell = 1}^{k-1})\right),
    \]
    where $U$ is a uniform distribution on $(0,1)^d$.
    It follows directly from this representation that $\kr_{\mu,\nu} \in \cplbc(\mu,\nu)$ and in particular $\KR_p \ge \AW_p$.

    Finally, the proof of Proposition \ref{prop:AW=KR} carries over to the multidimensional setting thanks to Corollary \ref{cor:Brenier.convergence}.
    Then the same reasoning as in the second proof of Theorem \ref{thm:Main} applies.
\end{proof}

\begin{proof}[Sketch of proof of Theorem \ref{thm:tildeKRandAW}]
    The first proof of Theorem \ref{thm:Main} directly carries over.
    Indeed, note that \eqref{eq:direct.proof.remark} holds by definition of $\widetilde{\KR}_p$ while all other steps did not use one dimensional ingredients.
\end{proof}

\subsection{On $\widetilde \KR_p$}\label{sec:OnWideTilde}

Let $p\in [1, \infty)$ and recall the definition of $\widetilde{\KR}_p$ from \eqref{eq:KRMultiTilde}.
Theorem \ref{thm:tildeKRandAW} asserts that $\widetilde{\KR}_p(\mu_n, \mu)\to 0$ if and only if $\AW_p(\mu_n, \mu)\to 0$. This follows from the same argument we used for the case $d=1$ in Section \ref{sec:ProofMainResult}. 
However, as the next lemma shows, $\widetilde{\KR}_p$ is just a semimetric but not a metric, i.e., the triangle inequality is not satisfied in general.

\begin{lemma} \label{lem:counterexample}
    Let $d, N \in \N,$ $d,N > 1$, and $p \ge 1$.
    Then the semimetric $\widetilde{\KR_p}$ on $\mathcal{P}_p(  (\R^d)^{N})$ does not satisfy the triangle inequality.
\end{lemma}
\begin{figure}
    \centering
    \begin{tikzpicture}[scale=1.50]
    
    \draw[gray, thick] (-1,0) -- (1,0);
    \draw[gray, thick] (0,-1) -- (0,1);
    \node at (0,1.2){$\proj^1_\# \mu$};
    \filldraw[black] (-0.8,0) circle (1pt) node[anchor=north]{(2,0)};
    \filldraw[black] (0.8,0) circle (1pt) node[anchor=north]{(-2,0)};
    
    \draw[gray, thick] (-1+2.5,0) -- (1+2.5,0);
    \draw[gray, thick] (0+2.5,-1) -- (0+2.5,1);
    \node at (0+2.5,1.2){$\proj^1_\# \nu$};
    \filldraw[black] (-0.1+2.5,0.8) circle (1pt) node[anchor=east]{(2,0)};
    \filldraw[black] (0.1+2.5,-0.8) circle (1pt) node[anchor=west]{(-2,0)};
    
    \draw[gray, thick] (-1+5,0) -- (1+5,0);
    \draw[gray, thick] (0+5,-1) -- (0+5,1);
    \node at (0+5,1.2){$\proj^1_\#\eta$};
    \filldraw[black] (0.1+5,0.8) circle (1pt) node[anchor=west]{(2,0)};
    \filldraw[black] (-0.1+5,-0.8) circle (1pt) node[anchor=east]{(-2,0)};
    \end{tikzpicture}
     \caption{Visualization of the support of the first $\mathbb R^2$-coodinates of the measures $\mu,\nu,\eta$ where the coordinates attached to each support point signify the second set of coordinates.}
    \label{fig:counterexample.1}
    
    \begin{tikzpicture}[scale=1.50]
    \draw[gray, thick] (-1,0-3) -- (1,0-3);
    \draw[gray, thick] (0,-1-3) -- (0,1-3);
    \node at (0,1.2-3){$T_{\mu,\nu}$};
    \filldraw[black] (-0.8,0-3) circle (1pt) node[anchor=north]{(2,0)};
    \filldraw[black] (0.8,0-3) circle (1pt) node[anchor=north]{(-2,0)};
    \filldraw[red] (-0.1,0.8-3) circle (1pt) node[anchor=east]{(2,0)};
    \filldraw[red] (0.1,-0.8-3) circle (1pt) node[anchor=west]{(-2,0)};
    \draw[blue, thick,->] (-0.8,0-3) parabola (-0.1,0.8-3);
    \draw[blue, thick,->] (+0.8,0-3) parabola (0.1,-0.8-3);

    \draw[gray, thick] (-1+2.5,0-3) -- (1+2.5,0-3);
    \draw[gray, thick] (0+2.5,-1-3) -- (0+2.5,1-3);
    \node at (0+2.5,1.2-3){$T_{\nu,\eta}$};
    \filldraw[black] (-0.1+2.5,0.8-3) circle (1pt) node[anchor=east]{(2,0)};
    \filldraw[black] (0.1+2.5,-0.8-3) circle (1pt) node[anchor=west]{(-2,0)};
    \filldraw[red] (0.1+2.5,0.8-3) circle (1pt) node[anchor=west]{(2,0)};
    \filldraw[red] (-0.1+2.5,-0.8-3) circle (1pt) node[anchor=east]{(-2,0)};
    \draw[blue, thick,->] (-0.1+2.5,0.8-3) parabola (0.1+2.5,0.8-3);
    \draw[blue, thick,->] (0.1+2.5,-0.8-3) parabola (-0.1+2.5,-0.8-3);

    \draw[gray, thick] (-1+5,0-3) -- (1+5,0-3);
    \draw[gray, thick] (0+5,-1-3) -- (0+5,1-3);
    \node at (0+5,1.2-3){$T_{\mu,\eta}$};
    \filldraw[black] (-0.8+5,0-3) circle (1pt) node[anchor=north]{(2,0)};
    \filldraw[black] (0.8+5,0-3) circle (1pt) node[anchor=north]{(-2,0)};
    \filldraw[red] (0.1+5,0.8-3) circle (1pt) node[anchor=west]{(2,0)};
    \filldraw[red] (-0.1+5,-0.8-3) circle (1pt) node[anchor=east]{(-2,0)};
    \draw[blue, thick,->] (-0.8+5,0-3) parabola (-0.1+5,-0.8-3);
    \draw[blue, thick,->] (+0.8+5,0-3) parabola (0.1+5,0.8-3);
    \end{tikzpicture}
    \caption{Illustration of the $\widetilde{\KR}_p$-optimal transport maps between $(\mu,\nu)$, $(\nu,\eta)$ and $(\mu,\eta)$. }
    \label{fig:counterexample.2}
\end{figure}

\begin{proof} 
    We prove this by providing a counterexample for $d=2$ and $N = 2$, and remark that the counterexample can be easily embedded into higher-dimensional spaces.
    Let $\epsilon > 0$, and define the probability measures $\mu,\nu,\eta \in \mathcal P_p(\mathbb R^{4})$ as
    \begin{align*}
    \mu:=&\frac{1}{2}\left(\delta_{((-1,0),(2,0))} +\delta_{((1,0),(-2,0))}\right),\\
    \nu:=&\frac{1}{2}\left(\delta_{((-\varepsilon,1),(2,0))} + \delta_{((\varepsilon,-1),(-2,0))}\right),\\
    \eta:=&\frac{1}{2}\left(\delta_{((\varepsilon,1),(2,0))} + \delta_{((-\varepsilon,-1),(-2,0))}\right).
    \end{align*}
These measures are visualised in Figures \ref{fig:counterexample.1} and \ref{fig:counterexample.2}.
Figure \ref{fig:counterexample.1} shows their support projected onto the first two coordinates. The labels of the points signify the second set of coordinates.
Between any of these measures there exists exactly one unique triangular optimal coupling, cf.\ \eqref{eq:KRDisDef.multidim}.
These couplings are of Monge-type and illustrated in Figure \ref{fig:counterexample.2}. 
We write $T_{\mu\nu}$ for the optimal map from $\mu$ to $\nu$. 
The red arrows indicate how the mass is moved.
By construction the coupling transports the first $\mathbb R^2$-coordinates optimally and thus only tries to minimize the movement of the points in the plane, ignoring the second set of $\mathbb R^2$-coordinates. 
Then the transport in the second dimension is predetermined by the transport in the plane since the conditional measures are all Diracs. 
We have
\begin{align*}
    T_{\mu\eta}&((-1,0),(2,0))=((-\varepsilon,-1),(-2,0)), \\
    T_{\mu\eta}&((1,0),(-2,0))=((\varepsilon,1),(2,0)),
\end{align*}
and get
\[
    \widetilde{\KR}_p(\mu,\eta) = (|1-\varepsilon|^p+ 1 + 4^p)^{\frac{1}{p}}\geq 4.
\]
Furthermore, we have
\begin{align*}
    T_{\mu\nu}&((-1,0),(2,0))=((-\varepsilon,1),(2,0)), \\
    T_{\mu\nu}&((1,0),(-2,0))=((\varepsilon,-1),(-2,0)),
\end{align*}
which implies that
\[
    \widetilde{\KR}_p(\mu,\nu)= (|1-\varepsilon|^p+1)^{\frac{1}{p}}.
\]
Finally, we have
\begin{align*}
    T_{\nu\eta}&((-\varepsilon,1),(2,0))=((\varepsilon,1),(2,0)), \\
    T_{\nu\eta}&((\varepsilon,-1),(-2,0))=(-\varepsilon,-1),(-2,0)),
\end{align*}  
whence,
\[ \widetilde{\KR}_p(\nu,\eta)=2\varepsilon. \]
Thus, when $\epsilon \le 1$ we find
\[
    \widetilde{\KR}_p(\mu,\nu)
    +
    \widetilde{\KR}_p(\nu,\eta)=((1-\varepsilon)^p+1)^{\frac{1}{p}}+2\varepsilon<4<\widetilde{\KR}_p(\mu,\eta),
\]
which violates the triangle inequality.
\end{proof}

\bibliographystyle{plain}
\bibliography{joint_biblio}

\end{document}